\documentclass[12pt]{amsart}
\usepackage{latexsym}
\usepackage{amssymb}
\usepackage{amscd}
\usepackage{mathrsfs}
\usepackage{rotating}
\renewcommand\a{\alpha}
\renewcommand\b{\beta}
\newcommand\g{\gamma}
\renewcommand\d{\delta}
\newcommand\la{\lambda}

\renewcommand\th{\theta}

\newcommand\x{\chi}

\newcommand\vf{\varphi}

\newcommand\Om{\Omega}
\newcommand\w{\omega}

\newcommand\vL{\varLambda}

\newcommand{\vT}{\varTheta}

\newcommand\ve{\varepsilon}

\newcommand\Ql{\bar{\mathbf Q}_l}

\newcommand\BQ{\mathbf Q}
\newcommand\BF{\mathbf F}

\newcommand\BR{\mathbf R}

\newcommand\BZ{\mathbf Z}

\newcommand\Bla{\boldsymbol\lambda}

\newcommand\Bmu{\boldsymbol\mu}
\newcommand\Bnu{\boldsymbol\nu}

\newcommand\CP{\mathcal{P}}

\newcommand\SB{\mathscr{B}}
\newcommand\SG{\mathscr{G}}
\newcommand\SM{\mathscr{M}}
 
\newcommand\SO{\mathscr{O}}
\newcommand\SP{\mathscr{P}}
\newcommand\SH{\mathscr{H}}

\newcommand\SX{\mathscr{X}}

\newcommand\FS{\mathfrak S}

\newcommand\Fu{\mathfrak u}
\newcommand\Fv{\mathfrak v}
\newcommand\Fg{\mathfrak g}

\newcommand\Fl{\mathfrak l}

\newcommand\iv{^{-1}}

\newcommand\wt{\widetilde}
\newcommand\wg{^{\wedge}}

\newcommand\ol{\overline}

\newcommand\hra{\hookrightarrow}
\newcommand\lra{\leftrightarrow}

\newcommand\IC{\operatorname{IC}}

\newcommand\End{\operatorname{End}}

\newcommand\Ind{\operatorname{Ind}}

\newcommand\Lie{\operatorname{Lie}}

\newcommand\nil{_{\operatorname{nil}}}

\newcommand\lp{\operatorname{\!\langle\!}}
\newcommand\rp{\operatorname{\!\rangle\!}}

\newcommand{\isom}{\,\raise2pt\hbox{$\underrightarrow{\sim}$}\,}
\numberwithin{equation}{section}

\newtheorem{thm}{Theorem}[section]
\newtheorem{lem}[thm]{Lemma}
\newtheorem{cor}[thm]{Corollary}
\newtheorem{prop}[thm]{Proposition}

\def \para#1{\par\medskip\textbf{#1}
              \addtocounter{thm}{1}}

\def \remark#1{\par\medskip\noindent
                \textbf{Remark #1}
                \addtocounter{thm}{1}}

\begin{document}
\setlength{\baselineskip}{4.9mm}
\setlength{\abovedisplayskip}{4.5mm}
\setlength{\belowdisplayskip}{4.5mm}
\renewcommand{\theenumi}{\roman{enumi}}
\renewcommand{\labelenumi}{(\theenumi)}
\renewcommand{\thefootnote}{\fnsymbol{footnote}}
\renewcommand{\thefootnote}{\fnsymbol{footnote}}
\allowdisplaybreaks[2]
\parindent=20pt

\pagestyle{myheadings}
\medskip
\begin{center}
 {\bf Double Kostka polynomials and Hall bimodule} 
\end{center}

\vspace{1cm}
\begin{center}
Liu Shiyuan and Toshiaki Shoji
\\
\vspace{0.7cm}
\title{}
\end{center}

\begin{abstract}
Double Kostka polynomials $K_{\Bla,\Bmu}(t)$ are polynomials in $t$, 
indexed by double partitions $\Bla, \Bmu$.  
As in the ordinary case, $K_{\Bla, \Bmu}(t)$
is defined in terms of Schur functions $s_{\Bla}(x)$ and Hall-Littlewood 
functions $P_{\Bmu}(x;t)$.  
In this paper, we study combinatorial properties of $K_{\Bla,\Bmu}(t)$ 
and $P_{\Bmu}(x;t)$.  In particular, we show that the 
Lascoux-Sch\"utzenberger type formula holds for $K_{\Bla,\Bmu}(t)$ in the case where
$\Bmu = (-;\mu'')$.    
Moreover, we show that the Hall bimodule $\SM$ introduced by Finkelberg-Ginzburg-Travkin
is isomorphic to the ring of symmetric functions (with two types of variables) and 
the natural basis $\Fu_{\Bla}$ of $\SM$ is sent to $P_{\Bla}(x;t)$ (up to scalar)
under this isomorphism. This gives an alternate approach for their result. 
\end{abstract}

\maketitle
\markboth{LIU AND SHOJI}{DOUBLE KOSTKA POLYNOMIALS}
\pagestyle{myheadings}

\begin{center}
{\sc Introduction}
\end{center}
\par\medskip

Kostka polynomials $K_{\Bla,\Bmu}(t)$, indexed by 
double partitions $\Bla, \Bmu$, were introduced in [S1, S2] as a generalization of
ordinary Kostka polynomials $K_{\la,\mu}(t)$ indexed by partitions $\la, \mu$.  
In this paper, we call them double Kostka polynomials.
Let $\vL = \vL(y)$ be the ring of symmetric functions with respect to the variables 
$y = (y_1, y_2, \dots)$ over $\BZ$. 
We regard $\vL\otimes \vL$ as the ring of 
symmetric functions $\vL(x^{(1)}, x^{(2)})$ with respect to 
two types of variables $x = (x^{(1)}, x^{(2)})$. 
Schur functions $\{s_{\Bla}(x)\}$ gives a basis of $\vL\otimes\vL$. In [S1, S2],
the function $P_{\Bmu}(x;t)$ indexed by a double partition $\Bmu$ 
was defined, as a generalization of the ordinary 
Hall-Littlewood function $P_{\mu}(x;t)$ indexed by a partition $\mu$. 
$\{P_{\Bmu}(x;t)\}$ gives a basis of $\BZ[t]\otimes_{\BZ}(\vL\otimes\vL)$, 
and as in the ordinary case, 
$K_{\Bla, \Bmu}(t)$ is defined as the coefficient of the transition matrix between 
two basis $\{ s_{\Bla}(x)\}$ and $\{P_{\Bmu}(x;t)\}$.  
\par
After the combinatorial introduction of $K_{\Bla, \Bmu}(t)$ in [S1, S2], 
Achar-Henderson [AH] gave a geometric interpretation of double Kostka polynomials 
in terms of the intersection cohomology associated to the closure of orbits in 
the enhanced nilpotent cone, which is a natural generalization of the classical 
result of Lusztig [L1]  that Kostka polynomials are interpreted by the intersection 
cohomology associated to the closure of nilpotent orbits in $\Fg\Fl_n$.
At the same time, Finkelberg-Ginzburg-Travkin [FGT] studied the convolution algeba 
associated to the affine Grassmannian in connection with double Kostka polynomials and 
the geometry of the enhanced nilpotent cone. In particular, they introduced the Hall 
bimodule $\SM$ (the mirabolic Hall bimodule in their terminology) as a generalization 
of the Hall algebra, and showed that $\SM$ is isomorphic to $\vL\otimes \vL$ over 
$\BZ[t, t\iv]$, and $P_{\Bla}(x;t)$ is obtained as the image of the natural basis 
$\Fu_{\Bla}$ of $\SM$. 
\par
In this paper, we study the combinatorial properties of $K_{\Bla,\Bmu}(t)$ and 
$P_{\Bmu}(x;t)$.  In particular, we show that the Lascoux-Sch\"utzenberger type formula
holds for $K_{\Bla, \Bmu}(t)$ in the case where $\Bmu = (-;\mu'')$ (Theorem 3.12).  
Moreover, in Theorem 4.7, we give a more direct proof for the above mentioned result
of [FGT] (in the sense that we don't appeal to the convolution algerba associated to 
the affine Grassmannian). 
\par
In the appendix, we give tables of double Kostaka polynomials for $2 \le n \le 5$, 
where $n$ is the size of double partitions.
The authors are grateful to J. Michel for the computer computation of those polynomials.

\section{Double Kostka polynomials} 

\para{1.1.}
First we recall basic properties of Hall-Littlewood functions and Kostka polynomials
in the original setting, following [M]. 
Let $\vL = \vL(y) = \bigoplus_{n \ge 0}\vL^n$ be the ring of symmetric functions 
over $\BZ$ with respect to the variables $y = (y_1, y_2, \dots)$, where $\vL^n$ denotes the 
free $\BZ$-module of symmetirc functions of degree $n$. 
We put $\vL_{\BQ} = \BQ\otimes_{\BZ}\vL$, $\vL^n_{\BQ} = \BQ\otimes_{\BZ}\vL^n$. 
For a partition $\la = (\la_1, \la_2, \dots, \la_k)$, put $|\la| = \sum_{i=1}^k\la_i$.
Let $\SP_n$ be the set of partitions of $n$, i.e. the set of $\la$ such that 
$|\la| = n$.
Let $s_{\la}$ be the Schur function associated to $\la \in \SP_n$.  Then 
$\{ s_{\la} \mid \la \in \SP_n \}$ gives a $\BZ$-baisis of $\vL^n$. 
Let $p_{\la} \in \vL^n$ be the power sum symmetric funtion associated to $\la$.
Then $\{ p_{\la} \mid \la \in \SP_n \}$ gives a $\BQ$-basis of 
$\vL^n_{\BQ}$. 
For $\la = (1^{m_1}, 2^{m_2}, \dots) \in \SP_n$, define an integer $z_{\la}$ by 
\begin{equation*}
\tag{1.1.1}
z_{\la} = \prod_{i \ge 1}i^{m_i}m_i!.
\end{equation*} 
Following [M, I], we introduce a scalar product on $\vL_{\BQ}$ by 
$\lp p_{\la}, p_{\mu} \rp = \d_{\la\mu}z_{\la}$.  
It is known that $\{s_{\la}\}$ form  an orthonormal basis of $\vL$. 

\para{1.2.}
Let $P_{\la}(y;t)$ be the Hall-Littlewood function associated to a partition $\la$.
Then $\{ P_{\la} \mid \la \in \SP_n \}$ gives a $\BZ[t]$-basis of 
$\vL^n[t] = \BZ[t]\otimes_{\BZ}\vL^n$, where $t$ is an indeterminate.
$P_{\la}$ enjoys a property that 
\begin{equation*}
\tag{1.2.1}
P_{\la}(y;0) = s_{\la},  \quad 
P_{\la}(y;1) = m_{\la},
\end{equation*}
where $m_{\la}(y)$ is a monomial symmetric function associated 
to $\la$.  
Kostka polynomials $K_{\la, \mu}(t) \in \BZ[t]$ ($\la, \mu \in \SP_n$) are defined by the formula
\begin{equation*}
\tag{1.2.2}
s_{\la}(y) = \sum_{\mu \in \CP_n}K_{\la,\mu}(t)P_{\mu}(y;t).
\end{equation*}  
\par
Recall the dominance order $\la \le \mu$ in $\SP_n$, which is defined by the condition 
$\la \le \mu$ if and only if $\sum_{j= 1}^i\la_j \le \sum_{j = 1}^i\mu_j$ for 
each $i \ge 1$. 
For each partition $\la = (\la_1, \dots, \la_k)$, we define an integer 
$n(\la)$ by $n(\la) = \sum_{i=1}^k(i-1)\la_i$. 
It is known that $K_{\la,\mu}(t) = 0$ unless $\la \ge \mu$, and that 
$K_{\la,\mu}(t)$ is a monic of degree $n(\mu) - n(\la)$ if $\la \ge \mu$
([M, III, (6.5)]).  

\par
For $\la = (\la_1, \dots, \la_k) \in \SP_n$ with $\la_k > 0$, 
we define $z_{\la}(t) \in \BQ(t)$ by 
\begin{equation*}
\tag{1.2.3}
z_{\la}(t) = z_{\la}\prod_{i \ge 1}(1 - t^{\la_i})\iv,
\end{equation*}
where $z_{\la}$ is as in (1.1.1). 
Following [M, III], we introduce a scalar product on $\vL_{\BQ}(t) = \BQ(t)\otimes_{\BZ}\vL$ by 
$\lp p_{\la}, p_{\mu} \rp = z_{\la}(t)\d_{\la,\mu}$.
Then $P_{\la}(y;t)$ form an orthogonal basis of $\vL[t] = \BZ[t]\otimes_{\BZ}\vL$.
In fact, they  are characterized by the following two properties
([M, III, (2.6) and (4.9)]);
\begin{equation*}
\tag{1.2.4}
P_{\la}(y;t) = s_{\la}(x) + \sum_{\mu < \la}w_{\la\mu}(t)s_{\mu}(x)
\end{equation*}
with $w_{\la\mu}(t) \in \BZ[t]$ , and
\begin{equation*}
\tag{1.2.5}
\lp P_{\la}, P_{\mu} \rp = 0 \text{ unless $\la = \mu$. } 
\end{equation*}

\para{1.3.}
Let  
$\Xi = \Xi(x) = \vL(x^{(1)})\otimes\vL(x^{(2)})$ be the ring of symmetric functions over $\BZ$ 
with respect to variables $x = (x^{(1)}, x^{(2)})$, where 
$x^{(1)} = (x^{(1)}_1, x^{(1)}_2, \dots), x^{(2)} = (x^{(2)}_1, x^{(2)}_2, \dots)$.
We denote it as $\Xi = \bigoplus_{n \ge 0}\Xi^n$, similarly to the case of $\vL$. 
  Let $\SP_{n,2}$ be the set of double 
partitions $\Bla = (\la', \la'')$ such that $|\la'| + |\la''| = n$. 
For $\Bla = (\la', \la'') \in \SP_{n,2}$, we define a Schur function
$s_{\Bla}(x) \in \Xi^n$ by 
\begin{equation*}
\tag{1.3.1}
s_{\Bla}(x) = s_{\la'}(x^{(1)})s_{\la''}(x^{(2)}).
\end{equation*}
Then 
$\{ s_{\Bla} \mid \Bla \in \SP_{n,2} \}$ gives a $\BZ$-basis of $\Xi^n$. 
For an integer $r \ge 0$, put 
$p_r^{(1)} = p_r(x^{(1)}) + p_r(x^{(2)})$, and 
$p_r^{(2)} = p_r(x^{(1)}) - p_r(x^{(2)})$, where $p_r$ is the power sum symmetric 
function in $\vL$. . 
For $\Bla \in \SP_{n,2}$, we define $p_{\Bla}(x) \in \Xi^n$ by 
\begin{equation*}
\tag{1.3.2}
p_{\Bla} = \prod_ip^{(1)}_{\la'_i}\prod_jp^{(2)}_{\la''_j},
\end{equation*}  
where $\Bla = (\la', \la'')$ such that $\la' = (\la'_1, \la'_2, \dots, \la'_{k'})$, 
$\la'' = (\la''_1, \la''_2, \dots, \la''_{k''})$ with $\la'_{k'}, \la''_{k''} > 0$. 
Then $\{ p_{\Bla} \mid \Bla \in \SP_{n,2} \}$ gives a $\BQ$-basis of 
$\Xi^n_{\BQ}$. 
For $\Bla \in \SP_{n,2}$, 
we define 
functions $z^{(1)}_{\Bla}(t), z^{(2)}_{\Bla}(t) \in \BQ(t)$ by  
\begin{equation*}
\tag{1.3.3}
z_{\Bla}^{(1)}(t) = \prod_{j = 1}^{k'}(1 - t^{\la_j'})\iv, \qquad
z_{\Bla}^{(2)}(t) = \prod_{j = 1}^{k''}(1 + t^{\la''_j})\iv.
\end{equation*}
For $\Bla \in \SP_{n,2}$, 
we define an integer $z_{\Bla}$ by $z_{\Bla} = 2^{k' + k''}z_{\la'}z_{\la''}$.  
We now define a function $z_{\Bla}(t) \in \BQ(t)$ by 
\begin{equation*}
\tag{1.3.4}
z_{\Bla}(t) = z_{\Bla}z_{\Bla}^{(1)}(t)z_{\Bla}^{(2)}(t).
\end{equation*}
Let $\Xi[t] = \BZ[t]\otimes_{\BZ}\Xi$ be the free $\BZ[t]$-module, 
and  $\Xi_{\BQ}(t) = \BQ(t)\otimes_{\BZ}\Xi$ be the $\BQ(t)$-space. 
Then $\{ p_{\Bla}(x) \mid \Bla \in \SP_{n,2} \}$ gives a basis of $\Xi^n_{\BQ}(t)$.
We define a scalar product on $\Xi_{\BQ}(t)$ by 
\begin{equation*}
\lp p_{\Bla}, p_{\Bmu} \rp = \d_{\Bla,\Bmu}z_{\Bla}(t).
\end{equation*} 
\par
We express a double partition $\Bla = (\la', \la'')$ as 
$\la' = (\la'_1, \dots, \la'_k), \la'' = (\la''_1, \dots, \la''_k)$ with some $k$, 
by allowing zero on parts $\la'_i, \la''_i$. We define a composition 
$c(\Bla)$ of $n$ by 
\begin{equation*}
c(\Bla) = (\la'_1, \la''_1, \la'_2, \la''_2, \dots, \la'_k, \la''_k).
\end{equation*} 
We define a partial order $\Bla \ge \Bmu$ on $\SP_{n,2}$ by the the condition 
$c(\Bla) \ge c(\Bmu)$, where $\ge $ is the dominance order on the set of 
compositions of $n$ defined in a similar way as in the case of partitions. 
\par
The following fact is known.

\begin{prop}[{[S1, S2]}] 
There exists a unique function $P_{\Bla}(x;t) \in \Xi_{\BQ}[t]$ satisfying the following properties.
\begin{enumerate}
\item
$P_{\Bla}$ is expressed as a linear combination of Schur functions $s_{\Bmu}$ as 
\begin{equation*}
P_{\Bla}(x;t) = s_{\Bla}(x) + \sum_{\Bmu < \Bla}u_{\Bla, \Bmu}(t)s_{\Bmu}(x) 
\end{equation*}
with $u_{\Bla, \Bmu}(t) \in \BQ(t)$.
\item
\ $\lp P_{\Bla}, P_{\Bmu} \rp = 0$ unless $\Bla = \Bmu$. 
\end{enumerate}
\end{prop}

\remark{1.5.}
$P_{\Bla}$ is called the Hall-Littlewood function associated to a double 
partiton $\Bla$.  More generally, Hall-Littlewood functions associated to 
$r$-partitions of $n$ was introduced in [S1]. However the arguments in [S1] 
is based on a fixed total order which is compatible with the partial order 
$\ge$ on $\SP_{n,2}$ even in the case of double partitions.  
In [S2, Theorem 2.8], the closed formula for $P_{\Bla}$ is given in the 
case of double partitions. This implies that $P_{\Bla}$ is independent of the 
choice of the total order, and is deterimned uniquely as in the above proposition. 
(The uniqueness of $P_{\Bla}$ also follows from the result of Achar-Henderson, 
see Theorem 2.4.)

\para{1.6.}
By Proposition 1.4, $\{ P_{\Bla} \mid \Bla \in \SP_{n,2} \}$ gives a basis 
of $\Xi_{\BQ}(t)$.  
For $\Bla, \Bmu \in \SP_{n,2}$, we define a function $K_{\Bla, \Bmu}(t) \in \BQ(t)$
by the formula

\begin{equation*}
s_{\Bla}(x) = \sum_{\Bmu \in \CP_{n,2}}K_{\Bla, \Bmu}(t)P_{\Bmu}(x;t).
\end{equation*} 
$K_{\Bla, \Bmu}(t)$ are called the Kostka functions associated to double partitions.
For each $\Bla = (\la', \la'') \in \SP_{n,2}$, put 
$n(\Bla) = n(\la' + \la'') = n(\la') + n(\la'')$.  We define an integer $a(\Bla)$  by 
\begin{equation*}
\tag{1.6.1}
a(\Bla) = 2n(\Bla) + |\la''|.
\end{equation*}
The following result was proved in [S2, Prop. 3.3]. 

\begin{prop}  
$K_{\Bla, \Bmu}(t) \in \BZ[t]$.  $K_{\Bla, \Bmu}(t) = 0$ unless 
$\Bla \ge \Bmu$.  If $\Bla \ge \Bmu$, 
$K_{\Bla, \Bmu}(t)$ is a monic of degree $a(\Bmu) - a(\Bla)$, 
hence $K_{\Bla,\Bla}(t) = 1$.  
In particular, $P_{\Bla}(x;t) \in \Xi^n[t]$, and $u_{\Bla,\Bmu}(t) \in \BZ[t]$. 
\end{prop}

\para{1.8.}
Since $K_{\Bla, \Bmu}(t)$ is a polynomial in $t$ associated to double partitions,
we call it the double Kostka polynomial.   
Put $\wt K_{\Bla, \Bmu}(t) = t^{a(\Bmu)}K_{\Bla, \Bmu}(t\iv)$. 
By Proposition 1.7, $\wt K_{\Bla, \Bmu}(t)$ is again contained in $\BZ[t]$, 
which we call the modified double Kostka polynomial.
In the case of Kostka polynomial $K_{\la,\mu}(t)$, we also put 
$\wt K_{\la,\mu}(t) = t^{n(\mu)}K_{\la, \mu}(t\iv)$. By 1.2, $\wt K_{\la, \mu}(t)$
is a polynomial in $\BZ[t]$, which is called the modified Kostka polynomial. 
\par
Following [S1, S2], we give a combinatorial characterization of $\wt K_{\la, \mu}(t)$
and $\wt K_{\Bla, \Bmu}(t)$. 
In order to discuss both cases simultaneously, we introduce some notation.
For $r = 1,2$, put $W_{n,r} = S_n \ltimes (\BZ/2\BZ)^n$. Hence $W_{n,r}$ is 
the symmetric group $S_n$ of degree $n$ if $r = 1$, and is the Weyl group of 
type $C_n$ if $r = 2$. 
For a (not necessarily irreducible) character $\x$ of $W_{n,r}$, we define the fake 
degree $R(\x)$ by
\begin{equation*}
\tag{1.8.1}
R(\x) = \frac{\prod_{i=1}^n(t^{ir}-1)}{|W_{n,r}|}
                     \sum_{w \in W_{n,r}}\frac{\ve(w)\x(w)}{\det_{V_0}(t-w)},
\end{equation*}
where $\ve$ is the sign character of $W_{n,r}$, and $V_0$ is the reflection representation 
of $W_{n,r}$ if $r = 2$ (i.e., $\dim V_0 = n$), and its restriction on $S_n$ if $r = 1$. 
Let $R(W_{n,r}) = \bigoplus_{i = 1}^NR_i$ be the coinvariant algebra over $\BQ$ 
associated to $W_{n,r}$, 
where $N$ is the number of positive roots of the root system of type $C_n$ 
(resp. type $A_{n-1}$) if $r = 2$ (resp. $r = 1$).  Then $R(W_{n,r})$ is a graded 
$W_{n,r}$-module, and we have 
\begin{equation*}
\tag{1.8.2}
R(\x) = \sum_{i=1}^N\lp \x, R_i\rp_{\,W_{n,r}} t^i,
\end{equation*} 
where $\lp \  , \ \rp_{\, W_{n,r}}$ is the inner product of characters of $W_{n,r}$. 
It follows that $R(\x) \in \BZ[t]$. 
It is known that irreducible characters of $W_{n,r}$ are parametrized by 
$\SP_{n,r}$ (we use the convention that $\SP_{n,1} = \SP_n$). 
We denote by $\x^{\Bla}$ the irreducible character of $W_{n,r}$ 
corresponding to $\Bla \in \SP_{n,r}$. 
(Here we use the parametrization such that the identity character corresponds to 
$\Bla = ((n), -)$ if $r = 2$, and  $\la = (n)$ if $r = 1$.)  
We define a square matrix $\Om = (\w_{\Bla, \Bmu})_{\Bla, \Bmu}$ by 
\begin{equation*}
\tag{1.8.3}
\w_{\Bla, \Bmu} = t^NR(\x^{\Bla} \otimes \x^{\Bmu} \otimes \ve).
\end{equation*}
We have the following result.
Note that Theorem 5.4 in [S1] is stated for a fixed total order on $\CP_{n,2}$. 
But in our case, it can be replaced by the partial order (see Remark 1.5). 

\begin{prop}[{[S1, Thm. 5.4]}]
Assume that $r = 2$.  There exist unique matrices 
$P = (p_{\Bla, \Bmu})$, $\vL = (\xi_{\Bla, \Bmu})$ over $\BQ[t]$ satisfying 
the equation 
\begin{equation*}
P\vL \,{}^t\!P = \Om,
\end{equation*}
subject to the condition that $\vL$ is a diagonal matrix and that 
\begin{equation*}
p_{\Bla, \Bmu} = \begin{cases}
                      0  &\quad\text{ unless $\Bmu \le \Bla$  }, \\ 
                     t^{a(\Bla)}  &\quad\text{ if $\Bla = \Bmu$ }.
                 \end{cases}
\end{equation*}
Then the entry $p_{\Bla, \Bmu}$ of the matrix $P$ coincides with 
$\wt K_{\Bla, \Bmu}(t)$. 
\par
A simialr result holds for the case $r = 1$ by replacing 
$\Bla, \Bmu \in \SP_{n,2}$
by $\la, \mu \in \SP_n$, and by replacing $a(\Bla)$ by $n(\la)$.
\end{prop}

\para{1.10.}
Assume that $\Bla = (-,\la'') \in \SP_{n,2}$. If $\Bmu \le \Bla$, 
then $\Bmu$ is of the form $\Bmu = (-, \mu'')$ with $\mu'' \le \la''$.  
Thus $\wt K_{\Bla, \Bmu}(t) = 0$ unless $\Bmu$ satisfies this condition. 
The following result was showm by Achar-Henderson [AH] by a geometric 
method (see Proposition 2.5 (ii)).  
We give below an alternate proof based on Proposition 1.9.

\begin{prop}  
Assume that $\Bla = (-, \la''), \Bmu = (-, \mu'') \in \SP_{n,2}$.
Then  
\begin{equation*}
\tag{1.11.1}
\wt K_{\Bla, \Bmu}(t) = t^n\wt K_{\la'', \mu''}(t^2).
\end{equation*}  
In particular, we have
\begin{equation*}
\tag{1.11.2}
K_{\Bla, \Bmu}(t) = K_{\la'',\mu''}(t^2).
\end{equation*}
\end{prop} 

\begin{proof}
(1.11.2) follows from (1.11.1).  We show (1.11.1).
We shall compute $\w_{\Bla, \Bmu}= t^NR(\x^{\Bla}\otimes \x^{\Bmu}\otimes \ve)$ 
for $\Bla = (-,\la''), \Bmu = (-.\mu'')$.
$\x^{\Bla}$ corresponds to the irreducible representation of $S_n$ with 
character $\x^{\la''}$, extended 
by the action of $(\BZ/2\BZ)^n$ such that any factor $\BZ/2\BZ$ acts non-trivially. 
This is the same for $\x^{\Bmu}$.  Hence $\x^{\Bla}\otimes \x^{\Bmu}$ corresponds to 
the representation of $S_n$ with character $\x^{\la''}\otimes \x^{\mu''}$, extended 
by the trivial action of $(\BZ/2\BZ)^n$.  
Thus $\x^{\Bla}\otimes \x^{\Bmu}\otimes \ve$ corresponds to the representation of $S_n$ 
with character $\x^{\la''}\otimes \x^{\mu''}\otimes \ve'$, extended by the action of 
$(\BZ/2\BZ)^n$ such that any factor $\BZ/2\BZ$ acts non-trivially, where 
$\ve'$ denote the sign character of $S_n$.  
Let $\{ s_1, \dots, s_n\}$ be the set of simple reflections of $W_n$.  
We identify the symmetric algebra $S(V_0^*)$ of $V_0$ with the polynomial ring 
$\BR[y_1, \dots, y_n]$ with the natural $W_n$ action, 
where $s_i$ permutes $y_i$ and $y_{i+1}$ ($1 \le i \le n-1$), and $s_n$ maps 
$y_n$ to $-y_n$.  Then $(\BZ/2\BZ)^n$-invariant subalgebra of $\BR[y_1, \dots, y_n]$
coincides with $\BR[y_1^2, \dots, y_n^2]$.  It follows that 
the $(\BZ/2\BZ)^n$-invariant subalgerba $R(W_n)^{(\BZ/2\BZ)^n}$ of 
$R(W_n)$ is isomorphic to $R(S_n)$ as graded algebras, where the degree $2i$-part of 
$R(W_n)^{(\BZ/2\BZ)^n}$ corresponds to the degree $i$ part of $R(S_n)$.     
Let $X$ be the subspace of $R(W_n)$ consisitng of vectors on which $(\BZ/2\BZ)^n$ 
acts in such a way that each factor $\BZ/2\BZ$ acts non-trivially.  Then 
$X = y_1\cdots y_n R(W_n)^{(\BZ/2\BZ)^n}$. 
It follows that 
\begin{equation*}
R(\x^{\Bla}\otimes\x^{\Bmu}\otimes\ve)(t) = 
        t^nR(\x^{\la''}\otimes\x^{\mu''}\otimes\ve')(t^2). 
\end{equation*}
Since $N = n^2$ for $W_n$-case, and $N = n(n-1)/2$ for $S_n$-case, this implies that 
\begin{equation*}
\tag{1.11.3}
\w_{\Bla, \Bmu}(t) = t^{2n}\w_{\la'', \mu''}(t^2)
\end{equation*}
We consider the embedding $\SP_n \hra \SP_{n,2}$ by $\la'' \mapsto (-, \la'')$. 
This embedding is compatible with the partial order of $\SP_n$ and $\SP_{n,2}$, 
and in fact, $\CP_n$ is identifed with the subset 
$\{ \Bmu \in \CP_{n,2} \mid \Bmu \le (-, (n))\}$ of $\SP_{n,2}$.  
We consider the matrix equation $P\vL\, {}^t\!P = \Om$ as in Proposition 1.9 for $r = 2$.  
Let $P_0, \vL_0, \Om_0$ be the submatrices of $P, \vL, \Om$ obtained  
by restricting the indices from $\SP_{n,2}$ to $\SP_n$. 
Then these matrices satisfy the relation $P_0\vL_0\,{}^t\!P_0 = \Om_0$.
By (1.11.3) $\Om_0$ coincides with $t^{2n}\Om'(t^2)$, where $\Om'$ denotes 
the matrix $\Om$ in the case $r = 1$.  
If we put $P' = t^{-n}P_0, \vL' = \vL_0$, we have a matrix equation 
$P'\vL'\,{}^t\!P' = \Om'(t^2)$. Note that the $(\la'',\la'')$-entry of $P'$ 
coincides with $t^{-n}t^{a(\Bla)} = t^{2n(\la'')}$.  Hence $P', \vL', \Om'$ 
satisfy all the requirements in Proposition 1.9 for the case $r= 1$, by 
replacing $t$ by $t^2$.   
Now by Proposition 1.9, we have $t^{-n}\wt K_{\Bla, \Bmu}(t) = \wt K_{\la'', \mu''}(t^2)$
as asserted. 
\end{proof}

As a corollary, we have 

\begin{cor}  
Assume that $\Bla = (-, \la'')$.  Then 
$P_{\Bla}(x;t) = P_{\la''}(x^{(2)};t^2)$. 
\end{cor}

\begin{proof}
Since $\Bla = (-,\la'')$, we have $s_{\Bla}(x) = s_{\la''}(x^{(2)})$. 
By (1.11.2), we have
\begin{equation*}
s_{\la''}(x^{(2)}) = \sum_{\mu'' \in \CP_n}K_{\la'',\mu''}(t^2)P_{\Bmu}(x; t).
\end{equation*}
We have  also
\begin{equation*}
s_{\la''}(x^{(2)}) = \sum_{\mu'' \in \CP_n}K_{\la'', \mu''}(t^2)P_{\mu''}(x^{(2)}; t^2).
\end{equation*}
Since 
$(K_{\la'', \mu''}(t^2))$ is a non-singular matrix indexed by $\CP_n$, 
the assertion follows.  
\end{proof}

\par\bigskip

\section{Geometric interpretation of double Kostka polynomials}

\para{2.1.}
In [L1], Lusztig gave a geometric interpretation of Kostka polynomials 
in terms of the intersection cohomology complex associated to the 
nilpotent orbits of $\Fg\Fl_n$.  
Let $V$ be an $n$-dimensional vector space over an algebraically closed 
field $k$, and put $G = GL(V)$.
Let  
$\Fg$ be the Lie algebra of $G$, and $\Fg\nil$ the nilpotent cone of $\Fg$.
$G$ acts on $\Fg\nil$ by the adjoint action, 
and the set of $G$-orbits in $\Fg\nil$ is in bijective correspondence 
with $\SP_n$ via the Jordan normal form of nilpotent elements. 
We denote by $\SO_{\la}$ 
the $G$-orbit corresponding to $\la \in \CP_n$. Let $\ol\SO_{\la}$ 
be the closure of $\SO_{\la}$ in $\Fg\nil$.   
Then we have $\ol\SO_{\la} = \coprod_{\mu \le \la}\SO_{\mu}$, where 
$\mu \le \la$ is the dominance order of $\SP_n$. Let 
$A_{\la} = \IC(\ol\SO_{\la}, \Ql)$ be the 
intersection cohomology complex of $\Ql$-sheaves, and $\SH^i_xA_{\la}$  
the stalk at $x \in \ol\SO_{\la}$ of the $i$-th cohomology sheaf $\SH^iA_{\la}$.
Lusztig's result is stated as follows.

\begin{thm}[{[L1, Thm. 2]}]   
$\SH^iA_{\la} = 0$ for odd $i$.  For each $x \in \SO_{\mu} \subset \ol\SO_{\la}$, 
\begin{equation*}
\wt K_{\la, \mu}(t) = t^{n(\la)}\sum_{i \ge 0}(\dim \SH^{2i}_xA_{\la})t^i.
\end{equation*}
\end{thm}

\para{2.3.}
The geometric interpretation of double Kostka polynomials analogous to 
Theorem 2.2 was established by Achar-Henderson [AH]. 
We follow the setting in 2.1.  Consider the direct product $\SX = \Fg \times V$, 
on which $G$ acts as $g: (x, v) \mapsto (gx, gv)$, where $gv$ is the natural action 
of $G$ on $V$.  Put $\SX\nil = \Fg\nil \times V$.  $\SX\nil$ is a $G$-stable subset of $\SX$, 
and is called the enhanced nilpotent cone. 
It is known by Achar-Henderson [AH] and by Travkin [T] that 
the set of $G$-orbits in $\SX\nil$ is in bijective correspondence with $\SP_{n,2}$. 
The correspondence is given as follows;  take $(x,v) \in \SX\nil$.  
Put $E^x = \{ g \in \End(V) \mid gx = xg \}$.  Then $W = E^xv$ is an $x$-stable 
subspace of $V$.  Let $\la'$ be the Jordan type of $x|_W$, and 
$\la''$ the Jordan type of $x|_{V/W}$.  Then $\Bla = (\la', \la'') \in \SP_{n,2}$, 
and the assignment $(x,v) \mapsto \Bla$ gives the required correspondence. 
We denote by $\SO_{\Bla}$ the $G$-orbit corresponding to $\Bla \in \SP_{n,2}$. 
The closure relation for $\SO_{\Bla}$ was described by [AH, Thm. 3.9] as follows;

\begin{equation*}
\tag{2.3.1}
\ol\SO_{\Bla} = \coprod_{\Bmu \le \Bla}\SO_{\Bmu},
\end{equation*}
where the partial order $\Bmu \le \Bla$ is the one defined in 1.3.
We consider the intersection cohomology complex $A_{\Bla} = \IC(\ol\SO_{\Bla}, \Ql)$ 
on $\SX\nil$ associated to $\Bla \in \SP_{n,2}$. 
The following result was proved by Achar-Henderson. 

\begin{thm}[{[AH, Thm. 5.2]}]   
Assume that $A_{\Bla}$ is attached to the enhanced nilpotent cone.  Then 
$\SH^iA_{\Bla} = 0$ for odd $i$.  For $z \in \SO_{\Bmu} \subset \ol\SO_{\Bla}$, 
\begin{equation*}
\wt K_{\Bla, \Bmu}(t) = t^{a(\Bla)}\sum_{i \ge 0}(\dim \SH^{2i}_zA_{\Bla})t^{2i}.
\end{equation*}
\end{thm}

Note that $\SH^{2i}$ corresponds to $t^{2i}$ in the enhanced case, which is
different from the correspondence $\SH^{2i} \lra t^i$ in the $\Fg\nil$ case.
As a corollary, we have

\begin{prop} [{[AH, Cor. 5.3]}]  
\begin{enumerate}
\item  $\wt K_{\Bla, \Bmu}(t) \in \BZ_{\ge 0}[t]$.  Moreover, 
only powers of $t$ congruent to $a(\Bla)$ 
modulo 2 occur in the polynomial.  
\item
Assume that $\Bla = (-, \la''), \Bmu = (-, \mu'')$.  Then 
$\wt K_{\Bla, \Bmu}(t) = t^n \wt K_{\la'', \mu''}(t^2)$.
\item
Assume that $\Bla = (\la',-)$ and $\Bmu = (\mu', \mu'')$.  Then 
$\wt K_{\Bla, \Bmu}(t) = \wt K_{\la', \mu' + \mu''}(t^2)$.  
\end{enumerate}
\end{prop} 

\begin{proof}
For the sake of completeness, we give the proof here. 
(i) is clear from the theorem. 
For (ii), take $\Bla = (-, \la'')$.  Then by the correspondence given
in 2.3, if $(x,v) \in \SO_{\Bla}$, then $v = 0$, and $x \in \SO_{\la''}$. 
It follows that $\SO_{\Bla} = \SO_{\la''}$ and that 
$A_{\Bla} \simeq A_{\la''}$.  $z \in \SO_{\Bmu}$ is also wtitten as 
$z = (x,0)$ with $x \in \SO_{\mu''}$.  
Then (ii) follows by comparing Theorem 2.2 and Theorem 2.4. 
For (iii), it was proved in [AH, Lemma 3.1] that 
$\ol \SO_{\Bla} = \ol\SO_{\la'} \times V$ for $\Bla = (\la',-)$. 
Thus $\IC(\ol\SO_{\Bla}, \Ql) \simeq \IC(\ol\SO_{\la'}, \Ql) \boxtimes (\Ql)_V$, 
where $(\Ql)_V$ is the constant sheaf $\Ql$ on $V$. 
It follows that $\SH^{2i}_zA_{\Bla} = \SH^{2i}_xA_{\la'}$ for $z = (x,v) \in \SO_{\Bmu}$.
Since $x \in \SO_{\mu' + \mu''}$, (iii) follows from Theorem 2.2 
(note that $a(\Bla) = 2n(\la')$). 
\end{proof}

\remark{2.6.}  Proposition 2.5 (ii) was also proved in Proposition 1.11 by a combinatorial
method.  We don't know whether (iii) is proved in a combinatorial way. 
However if we admit that $\wt K_{\Bla, \Bmu}(t)$ depends only on 
$\mu' + \mu''$ for $\Bla = (\la',-)$ (this is a consequence of (iii)), 
a similar argument as in the proof of Proposition 1.1 can be applied. 
 
\par\medskip
Proposition 2.5 (iii) implies the following.

\begin{cor}  
For $\nu \in \SP_n$, we have
\begin{equation*}
P_{\nu}(x^{(1)};t^2) = 
\sum_{\nu = \mu' + \mu''}t^{|\mu''|}P_{(\mu',\mu'')}(x;t).
\end{equation*}
\end{cor}

\begin{proof}
It follows from  Proposition 2.5 (iii) that 
$K_{\Bla, \Bmu}(t) = t^{|\mu''|}K_{\la', \mu' + \mu''}(t^2)$ for $\Bla = (\la', -)$. 
Since $s_{\Bla}(x) = s_{\la'}(x^{(1)})$, we have
\begin{align*}
s_{\la'}(x^{(1)}) &= \sum_{\Bmu \in \CP_{n,2}}K_{\Bla, \Bmu}(t)P_{\Bmu}(x;t)  \\
                  &= \sum_{\Bmu \in \CP_{n,2}}K_{\la', \mu' + \mu''}(t^2)    
                        t^{|\mu''|}P_{\Bmu}(x;t)  \\
                  &= \sum_{\nu \in \CP_n}K_{\la', \nu}(t^2)
                        \sum_{\nu = \mu' + \mu''}
                          t^{|\mu''|} P_{(\mu',\mu'')}(x;t).
\end{align*}
On the other hand, we have
\begin{equation*}
s_{\la'}(x^{(1)}) = \sum_{\nu \in \CP_n}K_{\la', \nu}(t^2)P_{\nu}(x^{(1)};t^2).
\end{equation*}
Since $(K_{\la',\nu}(t^2))$ is a non-singular matrix, 
we obtain the required formula. 
\end{proof}

\remark{2.8.}
The formula in Corollary 2.7 suggests that the behaviour of $P_{\Bmu}(x;t)$ at
$t = 1$ is different from that of ordinally Hall-Littlewood functions
given in (1.2.1).  In fact, 
by Corollary 1.12, $P_{(-,\nu)}(x;t) = P_{\nu}(x^{(2)}; t^2)$.
Hence $P_{(-,\nu)}(x;1) = m_{\nu}(x^{(2)})$ by (1.2.1). 
Also by (1.2.1) $P_{\nu}(x^{(1)};1) = m_{\nu}(x^{(1)})$.  
Then by Corollary 2.7, we have

\begin{equation*}
m_{\nu}(x^{(1)}) = m_{\nu}(x^{(2)}) + 
        \sum_{\nu = \mu' + \mu'', \mu' \ne \emptyset}P_{(\mu',\mu'')}(x;1).
\end{equation*} 
This formula shows that a certain cancelation occurs in the expression of $P_{\Bmu}(x;1)$ 
as a sum of monomials.  Concerning this, we will have a related result later
in Proposition 3.23.  

\para{2.9.}
There exists a geometric realization of double Kostka polynomials in terms 
of the exotic nilpotent cone instead of the enhanced nilpotent cone. 
Let $V$ be a $2n$-dimensional vector space over an algebraically closed field 
$k$ of odd characteristic.  Let $G = GL(V)$ and $\th$  an involutive automorphism 
of $G$ such that $G^{\th} = Sp(V)$. Put $H = G^{\th}$. Let $\Fg$ be the Lie algebra 
of $G$.  $\th$ induces a linear automorphism of order 2 on $\Fg$, which we denote 
also by $\th$.  $\Fg$ is decomposed as $\Fg = \Fg^{\th} \oplus \Fg^{-\th}$, 
where $\Fg^{\pm\th}$ is the eigenspace of $\th$ with eigenvalue $\pm 1$. 
Thus $\Fg^{\pm\th}$ are $H$-invariant subspaces in $\Fg$. We consider the direct 
product $\SX = \Fg^{-\th} \times V$, on which $H$ acts diagonally. 
Put $\Fg^{-\th}\nil = \Fg^{-\th} \cap \Fg\nil$.  Then $\Fg^{-\th}\nil$ is $H$-stable, 
and we consider $\SX\nil = \Fg^{-\th}\nil \times V$.  $\SX\nil$ is an $H$-invariant subset 
of $\SX$, and is called the exotic nilpotent cone.
It is known by Kato [K1] that the set of $H$-orbits in $\SX\nil$ is in bijective 
correspondence with $\SP_{n,2}$. We denote by $\SO_{\Bla}$ the $H$-orbit corresponding 
to $\Bla \in \SP_{n,2}$. It is also known by [AH] that the closure relations for 
$\SO_{\Bla}$ are given by the partial order $\le$ on $\SP_{n,2}$. 
We consider the intersection cohomology complex $A_{\Bla} = \IC(\ol\SO_{\Bla}, \Ql)$
on $\SX\nil$. 
The following result was proved by Kato [K2], and [SS2], independently.

\begin{thm}  
Assume that $A_{\Bla}$ is attached to the exotic nilpotent cone.  Then 
$\SH^iA_{\Bla} = 0$ unless $i \equiv 0 \pmod 4$.  
For $z \in \SO_{\Bmu} \subset \ol\SO_{\Bla}$, we have
\begin{equation*}
\wt K_{\Bla, \Bmu}(t) = t^{a(\Bla)}\sum_{i \ge 0}(\dim \SH^{4i}_zA_{\Bla})t^{2i}.
\end{equation*}  
\end{thm}

\para{2.11.}
Let $W_n$ be the Weyl group of type $C_n$.  
The advantage of the use of 
the exotic nilpotent cone relies on the fact that it has a good relationship 
with representations of Weyl groups, as explained below. Let $B$ be a $\th$-stable 
Borel subgroup of $G$.  Then $B^{\th}$ is a Borel subgrouop of $H$, and we denote 
by $\SB$ the flag variety $H/B^{\th}$ of $H$. 
Let $0 = M_0 \subset M_1 \subset \cdots \subset M_n \subset V$ be the (full) 
isotropic flag fixed by $B^{\th}$.  Hence $M_n$ is a maximal isotropic subspace of $V$.  
Put
\begin{equation*}
\wt\SX\nil = \{ (x, v, gB^{\th}) \in \Fg^{-\th}\nil \times V \times \SB
                     \mid g\iv x \in \Lie B, g\iv v \in M_n \},
\end{equation*}
and define a map $\pi_1 : \wt\SX\nil \to \SX\nil$ by $(x,v, gB^{\th}) \mapsto (x,v)$.
Then $\wt\SX\nil$ is smooth, irreducible and $\pi_1$ is proper surjective. 
Let  
$V_{\Bla}$ be the irreducible representation of $W_n$ corresponidng to $\x^{\Bla}$ 
($\Bla  \in \SP_{n,2}$). 
We consider the direct iamge $(\pi_1)_*\Ql$ of the constant sheaf $\Ql$ on $\wt\SX\nil$.
The following result is an analogue of the Springer correspondene for reductive groups, and 
was proved by Kato [K1], and [SS1], independently.

\begin{thm}  
$(\pi_1)_*\Ql[\dim \SX\nil]$ is a semisimple perverse sheaf on $\SX\nil$, equipped with 
$W_n$-action, and is decomposed as 

\begin{equation*}
\tag{2.12.1}
(\pi_1)_*\Ql[\dim \SX\nil] \simeq \bigoplus_{\Bla \in \SP_{n,2}}V_{\Bla} 
        \otimes A_{\Bla}[\dim \SO_{\Bla}],
\end{equation*}
where $A_{\Bla}[\dim \SO_{\Bla}]$ is a simple perverse sheaf on $\SX\nil$. 
\end{thm}

\para{2.13.}
For each $z = (x,v) \in \SX\nil$, put 

\begin{equation*}
\SB_z = \{ gB^{\th} \in \SB \mid g\iv x \in \Lie B, g\iv v \in M_n \}. 
\end{equation*}
$\SB_z$ is isomorphic to $\pi_1\iv(z)$, and is called the Springer fibre. 
Since $\SH^i_z((\pi_1)_*\Ql) \simeq H^i(\SB_z, \Ql)$, 
$H^i(\SB_z, \Ql)$ has a structure of $W_n$-module, which we call 
the Springer representation of $W_n$.
Put $K = (\pi_1)_*\Ql$.  By taking the stalk at $z \in \SX\nil$ of the $i$-th 
cohomology of both sides in (2.12.1),  
we have an isomorphism of $W_n$-modules, 
\begin{equation*}
\SH^i_zK \simeq H^i(\SB_z, \Ql) \simeq 
            \bigoplus_{\Bla \in \SP_{n,2}}V_{\Bla}\otimes 
\SH^{i + \dim \SO_{\Bla} - \dim \SX\nil}_zA_{\Bla}.
\end{equation*}
Since $\dim \SX\nil - \dim \SO_{\Bla} = 2a(\Bla)$ (see [SS2, (5.7.1)], this together with 
Theorem 2.10 imply the following result.

\begin{prop}  
Assume that $z \in \SO_{\Bmu}$.  Then 
$H^i(\SB_z, \Ql) = 0$ for odd $i$, and we have
\begin{equation*}
\wt K_{\Bla, \Bmu}(t) = \sum_{i \ge 0}\lp H^{2i}(\SB_z, \Ql), V_{\Bla}\,\rp_{\,W_n} t^i,
\end{equation*}
namely, the coefficient of $t^i$ in $\wt K_{\Bla, \Bmu}(t)$ is given by the 
multiplicity of $V_{\Bla}$ in the $W_n$-module $H^{2i}(\SB_z, \Ql)$. 
\end{prop}

\par\bigskip
\section{Combinatorial properties of $K_{\Bla,\Bmu}(t)$ and $P_{\Bmu}(x;t)$}

\para{3.1.}
In [AH], 
Achar-Henderson gave a formula expressing double Kostka polynomials in terms of
various ordinary Kostka polynomials.  We consider the enhanced nilpotent cone 
$\SX\nil = \Fg\nil \times V$ as in 2.3, under the assumption that $k$ is an 
algebraic closure of a finite field $\BF_q$.  Take $\Bmu, \Bnu \in \SP_{n,2}$. 
For each $z = (x,v) \in \SO_{\Bmu}$ and $\Bnu = (\nu',\nu'')$, we define 
a variety $\SG^{\Bmu}_{\Bnu}$ by

\begin{equation*}
\tag{3.1.1}
\begin{split}
\SG^{\Bmu}_{\Bnu} = \{W \subset &V \mid W  \text{ : $x$-stable subspace, }  v \in W, \\ 
               &x|_W \text{ type : $\nu'$, } 
                x|_{V/W} \text{ type : $\nu''$ }  \}.    
\end{split}
\end{equation*}
Note that if $z \in \SO_{\Bmu}(\BF_q)$, the variety $\SG^{\Bmu}_{\Bnu}$ is defined over 
$\BF_q$, and one can count the cardinality $|\SG^{\Bmu}_{\Bnu}(\BF_q)|$ of $\BF_q$-fixed 
points in $\SG^{\Bmu}_{\Bnu}$.  Clearly, $|\SG^{\Bmu}_{\Bnu}(\BF_q)|$ is independent of the
choice of $z \in \SO_{\Bmu}(\BF_q)$. 
\begin{prop}[{Achar-Henderson [AH, Prop. 5.8]}]  
Let $\Bmu, \Bnu \in \SP_{n,2}$.  
\begin{enumerate}
\item
There exists a polynomial $g_{\Bnu}^{\Bmu}(t) \in \BZ[t]$ such that 
$|\SG_{z, \Bnu}(\BF_q)| = g^{\Bmu}_{\Bnu}(q)$ for any finite field $\BF_q$ 
such that $z \in \SO_{\Bmu}(\BF_q)$.
\item 
Take $\Bla = (\la', \la''), \Bnu = (\nu', \nu'')$. 
Then we have
\begin{equation*}
\tag{3.2.1}
\wt K_{\Bla, \Bmu}(t) = t^{a(\Bla) - 2n(\Bla)}
      \sum_{\substack{\nu' \le \la'  \\ \nu'' \le \la''}}
           g^{\Bmu}_{\Bnu}(t^2)\wt K_{\la'\nu'}(t^2)\wt K_{\la''\nu''}(t^2).
\end{equation*}
\end{enumerate}
\end{prop}

\para{3.3.}
The formula (3.2.1) can be rewritten as 

\begin{equation*}
\tag{3.3.1}
K_{\Bla, \Bmu}(t) = t^{|\mu''| - |\la''|}\sum_{\Bnu = (\nu', \nu'') \in \SP_{n,2}}
                       t^{2n(\Bmu) - 2n(\Bnu)}g^{\Bmu}_{\Bnu}(t^{-2})
                            K_{\la'\nu'}(t^2)K_{\la''\nu''}(t^2).
\end{equation*}
Note that $g^{\Bmu}_{\Bnu}(t)$ is a generalization of Hall polynomials.
If $\Bmu = (-, \mu'')$, then $z = (x,v)$ with $v = 0$. In that case, 
$g^{\Bmu}_{\Bnu}(t)$ coincides with the original Hall polynomial 
$g^{\mu''}_{\nu'\nu''}(t)$ given in [M, II, 4].  In particular, 
if $g^{\mu}_{\nu'\nu''}(t) \ne 0$, then $g^{\mu}_{\nu'\nu''}(t)$ 
is a polynomial with degree $n(\mu) - n(\nu') - n(\nu'')$ and leading 
coefficient $c^{\mu}_{\nu'\nu''}$, where $c^{\mu}_{\nu'\nu''}$ is 
the Littlewood-Richardson coeffcient determined by the following conditions;
for partitions $\la, \mu,  \nu$, 
\begin{equation*}
\tag{3.3.2}
s_{\mu}s_{\nu} = \sum_{\la}c^{\la}_{\mu\nu}s_{\la}.
\end{equation*}
\par
For partitions $\la, \mu, \nu$, we define a polynomial $f^{\la}_{\mu\nu}(t)$ by 
\begin{equation*}
\tag{3.3.3}
P_{\mu}(y;t)P_{\nu}(y;t) = \sum_{\la}f^{\la}_{\mu\nu}(t)P_{\la}(y;t).
\end{equation*}
Then it is known by [M, III, (3.6)] that 
\begin{equation*}
\tag{3.3.4}
g^{\la}_{\mu\nu}(t) = t^{n(\la)-n(\mu) - n(\nu)}f^{\la}_{\mu\nu}(t^{-1}).
\end{equation*}
\par
We now assume that $\Bmu = (-, \mu'')$.  
Substituting (3.3.4) into (3.3.1), we have 

\begin{equation*}
\tag{3.3.5}
K_{\Bla, \Bmu}(t) = t^{|\la'|}\sum_{\nu',\nu''}f^{\mu''}_{\nu'\nu''}(t^2)
                       K_{\la'\nu'}(t^2)K_{\la''\nu''}(t^2).
\end{equation*}

\begin{lem}  
Assume that $\Bmu = (-,\mu'')$.  Then we have
\begin{align*}
\tag{3.4.1}
K_{\Bla,\Bmu}(t) &= t^{|\la'|}\sum_{\nu',\nu''}f^{\mu''}_{\nu'\nu''}(t^2)
                     K_{\la'\nu'}(t^2)K_{\la''\nu''}(t^2),  \\ 
\tag{3.4.2}
K_{\Bla,\Bmu}(t) &= t^{|\la'|}\sum_{\eta}c^{\eta}_{\la'\la''}K_{\eta,\mu''}(t^2).
\end{align*}
\end{lem}

\begin{proof}
The first equality is given in (3.3.5).  We show the second equality. 
One can write 

\begin{align*}
s_{\la'}(y) &= \sum_{\nu'}K_{\la'\nu'}(t)P_{\nu'}(y;t),  \\
s_{\la''}(y) &= \sum_{\nu''}K_{\la'',\nu''}(t)P_{\nu''}(y;t).
\end{align*}
Hence
\begin{align*}
\tag{3.4.3}
s_{\la'}(y)s_{\la''}(y) &= \sum_{\nu', \nu''}
                  K_{\la'\nu'}(t)K_{\la''\nu''}(t)P_{\nu'}(y;t)P_{\nu''}(y;t) \\
              &= \sum_{\nu',\nu''}\sum_{\mu''}f^{\mu''}_{\nu'\nu''}(t)
                     K_{\la'\nu'}(t)K_{\la''\nu''}(t)P_{\mu''}(y;t).
\end{align*}
On the other hand, 

\begin{align*}
\tag{3.4.4}
s_{\la'}(y)s_{\la''}(y) &= \sum_{\eta}c^{\eta}_{\la'\la''}s_{\eta}(y) \\
                        &= \sum_{\eta}c^{\eta}_{\la'\la''}
                             \sum_{\mu''}K_{\eta,\mu''}(t)P_{\mu''}(y;t).
\end{align*}
By comparing (3.4.3) and (3.4.4), we have, for each $\la', \la''$ and $\mu''$,  
\begin{equation*}
\sum_{\eta}c^{\eta}_{\la'\la''}K_{\eta,\mu''}(t) = \sum_{\nu', \nu''}
                    f^{\mu''}_{\nu'\nu''}(t)K_{\la'\nu'}(t)K_{\la''\nu''}(t). 
\end{equation*}
This proves the second equality. 
\end{proof}

\para{3.5.}
For $\la, \mu \in \SP_n$, let $SST(\la,\mu)$ be the set of semistandard tableaux 
of shape $\la$ and weight $\mu$. For a semistandard tableau $S$, the charge 
$c(S)$ is defined as in [M, III, 6]. Then Lascoux-Sch\"utzenberger theorem 
([M, III, (6.5)]) asserts that 

\begin{equation*}
\tag{3.5.1}
K_{\la\mu}(t) = \sum_{S \in SST(\la;\mu)}t^{c(S)}.
\end{equation*}  

In what follows, we shall prove an analogue of (3.5.1) for double Kostka polynomials 
$K_{\Bla, \Bmu}(t)$ for some special cases. 
Let $\Bla = (\la', \la'') \in \SP_{n,2}$.  A pair $T = (T_+, T_-)$ is called 
a semistandard tableau of shape $\Bla$ if $T_+$ (resp. $T_-$) is a semistandard tableau 
of shape $\la'$ (resp. $\la''$) with respect to the letters $1, \dots, n$. 
We denote by $SST(\la)$ the set of semistandard tableaux of shape $\Bla$. 
$T \in SST(\Bla)$ is regarded as a usual semistandard tableau associated to 
a skew diagram; write $\la' = (\la'_1, \la'_2, \dots, \la'_{k'})$ with $\la'_{k'} > 0$, 
and $\la'' = (\la''_1, \la''_2, \dots, \la''_{k''})$ with $\la''_{k''} > 0$. 
Put $a = \la''_1$.
We define a partition $\xi = (\xi_1, \dots, \xi_{k'+k''}) \in \SP_{n + ak'}$ by

\begin{equation*}
\xi_i = \begin{cases}
                \la'_i + a  &\quad\text{ for $1 \le i \le k'$ }, \\
                \la''_{i - k'}        &\quad\text{ for $k'+1 \le i \le k'+k''$. }
        \end{cases} 
\end{equation*}  
We define a partition $\th = (a^{k'})$  of rectangular shape. 
Then $\th \subset \xi$, and the skew diagram  $\xi - \th$ consist of 
connected component of shape $\la'$ and $\la''$. Thus $T \in SST(\Bla)$ 
can be identified with 
a semistandard tableau $\wt T$ of shape $\xi - \th$. 
Assume $\pi \in \SP_n$.  We say that $T \in SST(\la)$ has weight $\pi$ 
if the corresponding tableau $\wt T$ has shape $\xi-\th$ and weight $\pi$. 
We denote by $SST(\Bla, \pi)$ the set of semistandard tableau of shape 
$\Bla$ and weight $\pi$.  
\par
The set $SST(\Bla, \pi)$ is described as follows; for a partition 
$\nu \in \SP_m$ and $\a = (\a_1, \dots, \a_n) \in \BZ^n_{\ge 0}$ such that
$|\a| = \sum_i \a_i = m$, 
let $SST(\nu;\a)$ be the set of semistandard tableau of shape $\nu$ and 
weight $\a$. Then we have

\begin{equation*}
\tag{3.5.2}
SST(\Bla,\pi) = \coprod_{\substack{\a + \b = \pi \\
                                   |\a| = |\la'|}}
                           (SST(\la',\a) \times SST(\la'', \b)).
\end{equation*}  

\remark{3.6.}
Usually, the weight of a semistandard tableau is assumed to be a partition.
Here we need to consider the weight which is not a partition.  But this gives 
no essential difference. In fact, we consider the set $SST(\nu;\a)$. 
$S_n$ acts on $\BZ^n_{\ge 0}$ by a permutation of factors.  
We denote by $O(\a)$ the $S_n$-orbit of 
$\a$ in $\BZ^n_{\ge 0}$.   There exists a unique $\mu \in O(\a)$ 
such that $\mu$ is a partition. Then we have  
$|SST(\nu;\a)| = |SST(\nu; \mu)|$.  
This follows from  (5.12) in [M, I] and the discussion below  (though 
it is not written explicitly).

\para{3.7.}
For (an ordinary) semistandard tableau $S$, a word $w(S)$ is defined 
as a sequence of letters $1, \dots, n$, reading from right to left, 
and top to down. This definition works for the semistandard tableau 
assoicaited to a skew-diagram. 
For a semistandard tableau $T = (T_+, T_-) \in SST(\Bla)$, we define 
the associated word $w(T)$ by $w(T) = w(T_+)w(T_-)$. 
Hence $w(T)$ coinicdes with $w(\wt T)$.
\par 
Following [M, I, 9], we introduce a notion 
of lattice permutation.  A word $w = a_1\cdots a_N$ consisting of letters 
$1, \dots, n$ is called a lattice 
permutation if for $1 \le r \le N$ and $1 \le i \le n-1$, the number of 
occurrences of the letter $i$ in $a_1\cdots a_r$ is $\ge$ the number of 
occurrences of the letter $i+1$.    
We denote by $SST^0(\Bla, \pi)$ the set of semistandard tableau 
$T \in SST(\Bla, \pi)$ such that $w(T)$ is a lattice permutation.

\begin{lem} 
Assume that $\Bla \in \SP_{n,2}$, $\pi \in \SP_n$.  
There exists a bijective map 
\begin{equation*}
\tag{3.8.1}
\vT : SST(\Bla, \pi) \isom \coprod_{\nu \in \SP_n} (SST^0(\Bla, \nu) \times SST(\nu, \pi))
\end{equation*}
\end{lem}

\begin{proof}
Under the correspondence $T \lra \wt T$ in 3.5, the set $SST(\Bla,\pi)$   
can be identified with the set $SST(\xi -\th, \pi)$.   
Then (3.8.1) is a special case of the bijection given in [M, I, (9.4)].
In (9.4), this bijection is explicitly constructed.  
\end{proof}

\begin{cor}  
Assume that $\Bla = (\la', \la'') \in \SP_{n,2}$, $\nu \in \SP_n$.  Then we have
\begin{equation*}
|SST^0(\Bla,\nu)| = c^{\nu}_{\la', \la''}.
\end{equation*}
\end{cor}

\begin{proof}
We prove the corollary by modifying the discussion in [M, I, 9].
By [M, I, (5.12)], we have

\begin{align*}
s_{\la'}(y) &= \sum_{S' \in SST(\la')}y^{S'}, \\
s_{\la''}(y) &= \sum_{S'' \in SST(\la'')}y^{S''}.
\end{align*}
It follows that 
\begin{equation*}
s_{\la'}(y)s_{\la''}(y) = \sum_{T \in SST(\Bla)}y^{T}.
\end{equation*}
By a similar argument as in the proof of (5.14) in [M, I], we have 
\begin{equation*}
|SST(\Bla,\pi)| = \lp s_{\la'}s_{\la''}, h_{\pi}\rp, 
\end{equation*}
where $h_{\pi}$ is a complete symmetric function assoicated to 
$\pi$. 
Similarly, we have $|SST(\nu,\pi)| = \lp s_{\nu}, h_{\pi}\rp$. 
Then by (3.8.1), we have
\begin{equation*}
\lp s_{\la'}s_{\la''}, h_{\pi}\rp = \sum_{\nu \in \SP_n}|SST^0(\Bla,\nu)|
         \lp s_{\nu}, h_{\pi}\rp 
\end{equation*}
for any $\pi \in \SP_n$.  It follows that 
\begin{equation*}
\tag{3.9.1}
s_{\la'}s_{\la''} = \sum_{\nu \in \SP_n}|SST^0(\Bla,\nu)|s_{\nu}.
\end{equation*}
On the other hand, by (3.3.2) we have
\begin{equation*}
\tag{3.9.2}
s_{\la'}s_{\la''} = \sum_{\nu\in \SP_n}c^{\nu}_{\la', \la''}s_{\nu}.
\end{equation*}
By comparing the coefficient of $s_{\nu}$ in (3.9.1) with (3.9.2), 
we obtain the result. 
\end{proof}

\remark{3.10.}
The Littlewood-Richardson rule is a combinatorial procedure of computing 
Littlewood-Richardson coefficients.  In [M, I, (9.2)]
it is stated that $c^{\nu}_{\la',\la''}$ is equal to the number of smsitandard 
tableaux $T$ of shape $\nu-\la'$ and weight $\la''$ such that $w(T)$ is a lattice
permutation. Hence 
Corollary 3.9 gives a variant of the Littlewood-Richardson rule.

\para{3.11.}
Assume that $\Bla \in \SP_{n,2}$ and $\mu'' \in \SP_n$.  
For $T \in SST(\Bla,\mu'')$, write $\vT(T) = (D, S)$, with $S \in SST(\nu, \mu'')$ 
for some $\nu$.  We define a charge $c(T)$ of $T$ by $c(T) = c(S)$, where 
$c(S)$ is the charge of $S$ as in (3.5.1).  The following formula is an analogue of 
Lascoux-Sch\"utzenberger thorem  for the double Kostka polynomial $K_{\Bla, \Bmu}(t)$
in the case where $\Bmu = (-, \mu'')$.

\begin{thm}  
Let $\Bla, \Bmu \in \SP_{n,2}$, and assume that $\Bmu = (-,\mu'')$.  Then 

\begin{equation*}
K_{\Bla,\Bmu}(t) = t^{|\la'|}\sum_{T \in SST(\Bla, \mu'')}t^{2c(T)}.
\end{equation*} 
\end{thm}

\begin{proof}
We define a map $\Psi : SST(\Bla,\mu'') \to \coprod_{\nu \in \SP_n}SST(\nu, \mu'')$ 
by $T \mapsto S$, where $\vT(T) = (D, S)$. Then by Corollary 3.9, for each 
$S \in SST(\nu,\mu'')$, the set $\Psi\iv(S)$ has the cardinality $c^{\nu}_{\la'\la''}$, 
and, by definition,  any element $T \in \Psi\iv(S)$ has the charge $c(T) = c(S)$.  Hence

\begin{align*}
\sum_{T \in SST(\Bla,\mu'')}t^{c(T)} 
       &= \sum_{\nu \in \SP_n}\sum_{S \in SST(\nu,\mu'')}c^{\nu}_{\la'\la''}t^{c(S)}  \\
       &= \sum_{\nu \in \SP_n}c^{\nu}_{\la'\la''}K_{\nu,\mu''}(t)
\end{align*}
since $K_{\nu, \mu''}(t) = \sum_St^{c(S)}$ by (3.5.1). 
Now the theorem follows from  (3.4.2).  
\end{proof}

\begin{cor}  
Assume that $\Bla, \Bmu \in \SP_{n,2}$ with $\Bmu = (-,\mu'')$.  Then we have

\begin{equation*}
K_{\Bla,\Bmu}(1) = |SST(\Bla,\mu'')|.
\end{equation*}
\end{cor}

\para{3.14.}
Here we recall the explicit  
construction of $\x^{\Bla}$ for $\Bla = (\la',\la'') \in \SP_{n,2}$. 
Put $|\la'| = m', |\la''| = m''$.  Let 
$\x^{\la'}$ (resp. $\x^{\la''}$) be the irreducible character of $S_{m'}$ 
(resp. $S_{m''}$) corresponding to the partition $\la' \in \SP_{m'}$
(resp. $\la'' \in \SP_{m''}$).
We denote by $\wt \x^{\la'}$ the irreducible character of 
$W_{m'} = S_{m'} \ltimes (\BZ/2\BZ)^{m'}$ obtained by extending $\x^{\la'}$ 
by the trivial action of $(\BZ/2\BZ)^{m'}$. We also denote by $\wt \x^{\la''}$
the irreducible character of $W_{m''} = S_{m''}\ltimes (\BZ/2\BZ)^{m''}$ 
by extending $\x^{\la''}$ by defining the action of $(\BZ/2\BZ)^{m''}$ so that
each factor $\BZ/2\BZ$ acts non-trivially. Then 
$\Ind_{W_{m'} \times W_{m''}}^{W_n}\wt \x^{\Bla'} \otimes \wt \x^{\Bla''}$ 
gives an irreducible character $\x^{\Bla}$. 
It follows from the construction that $\x^{\Bla}|_{S_n}$ coincides with 
$\Ind_{S_{m'} \times S_{m''}}^{S_n}\x^{\la'}\otimes \x^{\la''}$. 
\par 
For $\nu = (\nu_1, \dots, \nu_k) \in \SP_n$, we denote by $S_{\nu}$ the Young subgroup
$S_{\nu_1} \times \cdots \times S_{\nu_k}$. 
We show the following formula.

\begin{prop}  
Let $\Bla, \Bmu \in \SP_{n,2}$ with $\Bmu = (-, \mu'')$.  Then we have
\begin{equation*}
\tag{3.15.1}
K_{\Bla, \Bmu}(1) = \lp \Ind_{S_{\mu''}}^{W_n}1,\x^{\Bla} \rp_{\,W_n}.
\end{equation*}
\end{prop}

\begin{proof}
Under the notation in 3.14, we compute the inner product.
\begin{align*}
\lp \Ind_{S_{\mu''}}^{W_n}1, \x^{\Bla} \rp_{\, W_n} 
       &= \lp \Ind_{S_{\mu''}}^{S_n}1, \x^{\Bla}|_{S_n} \rp _{\, S_n} \\
       &= \lp \Ind_{S_{\mu''}}^{S_n}1, \Ind_{S_{m'} 
            \times S_{m''}}^{S_n}\x^{\la'}\otimes \x^{\la''} \rp_{\, S_n}.
\end{align*}
Here we can write 
$\Ind_{S_{m'}\times S_{m''}}^{S_n}\x^{\la'}\otimes \x^{\la''}
            = \sum_{\nu \in \SP_n}c^{\nu}_{\la'\la''}\x^{\nu}$ 
by using the Littlewood-Richardson coefficients.
Thus

\begin{equation*}
\lp \Ind_{S_{\mu''}}^{W_n}1, \x^{\Bla} \rp_{\,W_n} 
             = \sum_{\nu \in \SP_n}c^{\nu}_{\la'\la''}
                 \lp \Ind_{S_{\mu''}}^{S_n}1, \x^{\nu}\rp_{\, S_n}. 
\end{equation*} 
But it is known that $\lp \Ind_{S_{\mu''}}^{S_n}1, \x^{\nu}\rp_{\, S_n} 
              = K_{\nu, \mu''}(1)$ (see eg. [M, I, Remark after (7.8)]).
Hence we have

\begin{equation*}
\lp \Ind_{S_{\mu''}}^{W_n}1, \x^{\Bla} \rp_{\, W_n} 
                 = \sum_{\nu \in \SP_n}c^{\nu}_{\la'\la''}K_{\nu, \mu''}(1).
\end{equation*}
Then the proposition follows from (3.4.2), by substituting 
$t = 1$. 
\end{proof}

\begin{cor}  
Let $\Bmu = (-; \mu'')$.  Then  for $z \in \SO_{\Bmu}$, we have 
\begin{equation*}
\tag{3.16.1}
\bigoplus_{i \ge 0}H^{2i}(\SB_z, \Ql) \simeq \Ind_{S_{\mu''}}^{W_n}1
\end{equation*}
as $W_n$-modules.
\end{cor}

\begin{proof}
Put $H^*(\SB_z) = \bigoplus_{i \ge 0}H^{2i}(\SB_z, \Ql)$. 
Then Proposition 2.14 shows that 
\begin{equation*}
K_{\Bla,\Bmu}(1) = \lp H^*(\SB_z), \x^{\Bla}\rp_{\, W_n}
\end{equation*}
for any $\Bla \in \SP_{n,2}$. 
Thus, by comparing it with (3.15.1), we obtain the required 
formula.
\end{proof}

\remark{3.17.}
It would be interesting to compare (3.16.1) with a similar formula for 
the ordinary Springer representations of type $C_n$. We follow the 
setting in 2.11.  For $x \in \Fg\nil^{\th}$, we define 

\begin{equation*}
\SB^{\star}_x = \{ gB^{\th} \in \SB \mid  g\iv x \in \Lie B^{\th} \}. 
\end{equation*} 
$\SB_x^{\star}$ is the original Springer fibre associated to $x \in \Fg^{\th}\nil$, 
and the cohomology group $H^i(\SB^{\star}_x, \Ql)$ has a natural action of $W_n$. 
It is known that $H^i(\SB_x^{\star}, \Ql) = 0$ for odd $i$.  Let $\Fl^{\th}$ be a
Levi subalgebra of a parabolic subalgebra of $\Fg^{\th}$ of type 
$A_{\mu''_1-1} + A_{\mu''_2-1} + \cdots + A_{\mu''_k-1}$ for 
$\mu'' = (\mu''_1, \mu''_2, \dots, \mu''_k) \in \SP_n$. Assume that $x$ is a regualr nilpotent 
element in $\Fl^{\th}\nil$. Then by a general formula due to [L2], we have 
\begin{equation*}
\tag{3.17.1}
\bigoplus_{i \ge 0}H^{2i}(\SB^{\star}_x, \Ql) \simeq \Ind_{S_{\mu''}}^{W_n}1
\end{equation*}
as $W_n$-modules.  However, the graded $W_n$-module structures in (3.16.1) and 
(3.17.1) do not coincide in general. For example, assume that $n = 2$, and 
$\Bmu = (-;2)$, i.e., $\mu'' = (2)$.  In that case, 
$\Ind_{S_{\mu''}}^{W_2}1 = \Ind_{S_2}^{W_2}1 =  \x^{(-;2)} + \x^{(1;1)} + \x^{(2;-)}$. 
We have

\begin{align*}
H^4(\SB_z,\Ql) &= \x^{(-;2)}, \quad H^2(\SB_z, \Ql) = \x^{(1;1)}, 
                             \quad H^0(\SB_z, \Ql) = \x^{(2;-)}, \\
H^2(\SB^{\star}_x, \Ql) &= \x^{(-;2)} + \x^{(1;1)}, \quad 
                H^0(\SB^{\star}_x, \Ql) = \x^{(2;-)}.
\end{align*}

\para{3.18.}
We shall give an interpretation of the formula (3.2.1) in terms of the 
Springer modules. 
Let $A_n = (\BZ/2\BZ)^n$ be the abelian subgroup of $W_n$. 
We denote by $t_1, \dots, t_n$ the generators of $A_n$, where $t_i$ is 
the generator of the $i$-th component $\BZ/2\BZ$. 
Let $\vf$ be a linear character of $A_n$.  For each $A_n$-module $X$, we 
denote by $X_{\vf}$ the weight space of $X$ corresponding to $\vf$, namely
$X_{\vf} = \{ v \in X \mid av = \vf(a)v \text{ for } a \in A_n\}$. 
Let $S_{\vf}$ be the stabilizer of $\vf$ in $S_n$ under the action of $S_n$ 
on $A_n$. Then $S_{\vf} \simeq S_m \times S_{n-m}$, where $m$ is the number of 
$i$ such that $\vf(t_i) = 1$. If $X$ is an $W_n$-module, $X$ is an $A_n$-module
by restriction.  Then $X_{\vf}$ turns out to be an $S_{\vf}$-module. 
\par
The $W_n$-module $H^i(\SB_z, \Ql)$, which is called the (exotic) Springer modue, 
 is isomorphic to each other for 
$z \in \SO_{\Bmu}$ ($\Bmu \in \SP_{n,2}$).
In the discussion below, we denote it simply by $H^i(\SB_{\Bmu})$.  
Let $\SB^0 = G_0/B_0$ be the flag variety of $G_0 = GL_n$, where $B_0$ is a Borel subgroup of 
$G_0$. Recall that for each nilpotent element $x \in \Fg\Fl_n$, the Springer fibre 
$\SB^0_x$ is defined as

\begin{equation*}
\SB^0_x = \{ gB_0 \in \SB^0 \mid g\iv x \in \Lie B_0 \},
\end{equation*} 
and the cohomology group $H^i(\SB^0_x, \Ql)$ has a natural strucutre of $S_n$-module, 
the Springer module.
Since the $S_n$-module strucutre does not depend on $x \in \SO_{\nu}$ ($\nu \in \SP_n$), 
we denote it by $H^i(\SB^0_{\nu})$. 
Let $A_n\wg$ be the set of irreducible characters of $A_n$.  Then we have the weight space
decomposition 

\begin{equation*}
H^i(\SB_{\Bmu}) = \bigoplus_{\vf \in A_n\wg}H^i(\SB_{\Bmu})_{\vf},
\end{equation*}
where each $H^i(\SB_{\Bmu})_{\vf}$ has a structure of $S_{\vf}$-module.
\par
Recall the polynomial $g^{\Bmu}_{\Bnu}(t) \in \BZ[t]$ for $\Bmu, \Bnu \in \SP_{n,2}$
given Proposition 3.2.  We write it as 

\begin{equation*}
g^{\Bmu}_{\Bnu}(t) = \sum_{\ell \ge 0}g^{\Bmu}_{\Bnu, \ell}t^{\ell}
\end{equation*}
with (possibly negative) integers $g^{\Bmu}_{\Bnu,\ell}$. 
The following proposition gives a description of $H^i(\SB_{\Bmu})_{\vf}$ 
in terms of the Springer modules of $S_{\vf}$.  

\begin{prop}  
Assume that $\Bmu \in \SP_{n,2}$.  Let $\vf \in A_n\wg$ be such that 
$S_{\vf} \simeq S_m \times S_{n-m}$.  Then the following 
equality holds (in the Grothendieck group of the category of $S_{\vf}$-modules)

\begin{equation*}
H^{2i}(\SB_{\Bmu})_{\vf} = \sum_{\substack{\Bnu = (\nu', \nu'') \in \SP_{n,2} \\
                                   |\nu'| = m}}
              \sum_{j,k,\ell}g^{\Bmu}_{\Bnu, \ell}\bigl(H^{2j}(\SB^0_{\nu'})
                            \otimes H^{2k}(\SB^0_{\nu''})\bigr),
\end{equation*}
where the second sum is taken over all $j, k, \ell$ satisfying the condition 
\begin{equation*}
i = (n-m) + 2\ell + 2(j + k).
\end{equation*}
\end{prop}

\begin{proof}
By Proposition 2.14, one can write (as an identity in the Grothendieck group of the category 
of $S_{\vf}$-modules, extended by scalar to $\BZ[t]$)

\begin{equation*}
\tag{3.19.1}
\sum_{i \ge 0}H^{2i}(\SB_{\Bmu})_{\vf} t^i \simeq  
       \sum_{\Bla \in \SP_{n,2}}\wt K_{\Bla, \Bmu}(t)(V_{\Bla})_{\vf} 
\end{equation*}
for each $\vf \in A_n\wg$. 
Assume that $S_{\vf} \simeq S_m \times S_{n-m}$.
It follows from the explicit construction of $V_{\Bla}$ in 3.14 that 
$(V_{\Bla})_{\vf} = 0$ unless $|\la'| = m, |\la''| = n-m$, and in that 
case, $(V_{\Bla})_{\vf} \simeq V_{\la'}\otimes V_{\la''}$ as 
$S_m \times S_{n-m}$-modules, where $V_{\la'}$ denotes the irreducible $S_m$-module
corresponding to $\x^{\la'}$, and similarly for $V_{\la''}$.
By (3.2.1), the right hand side of (3.19.1) can be written as 

\begin{align*}
&t^{n-m}\sum_{\substack{\la' \in \SP_m \\
                \la'' \in \SP_{n-m}}}\sum_{\Bnu = (\nu', \nu'') \in \SP_{n,2}}g^{\Bmu}_{\Bnu}(t^2)
        \wt K_{\la',\nu'}(t^2)\wt K_{\la'', \nu''}(t^2)V_{\la'}\otimes V_{\la''}  \\
   &= t^{n-m}\sum_{\Bnu}g^{\Bmu}_{\Bnu}(t^2)
           \biggl(\sum_{\la' \in \SP_m}\wt K_{\la',\nu'}(t^2)V_{\la'}\biggr)
              \otimes \biggl(\sum_{\la'' \in \SP_{n-m}}\wt K_{\la'',\nu''}(t^2)V_{\la''}\biggr)  \\
   &= t^{n-m}\sum_{\Bnu}g^{\Bmu}_{\Bnu}(t^2)
             \biggl(\sum_{i \ge 0}H^{2i}(\SB^0_{\nu'})t^{2i}\biggr) \otimes 
             \biggl(\sum_{i \ge 0}H^{2i}(\SB^0_{\nu''})t^{2i}\biggr),   
\end{align*}   
where the last equality follows from the formulas analogous to Proposition 2.14 in the case 
of $GL_m$ and $GL_{n-m}$. 
By comparing the last expression with the left hand side of (3.19.1), we obtain 
the proposition. 
\end{proof}

\para{3.20.}
We consider $\vf \in A_n\wg$ in the special case where $m = n$ or $m = 0$.
Put $\vf = \vf_+$ (resp. $\vf = \vf_-$) if $m = n$ (resp. $m = 0$). 
In these cases, $S_{\vf} \simeq S_n$, and we have a more precise description 
of the $S_n$-module $H^i(\SB_{\Bmu})_{\vf}$ as follows.
(Note that $H^i(\SB_{\Bmu})_{\vf_+}$ coincides with the $A_n$-fixed point subspace of
$H^i(\SB_{\Bmu})$.  The formula (i) in the corollary should be compared with the result
in [SSr], where the case of ordinary Springer representations of type $C_n$ is discussed.)  

\begin{cor}  
Assume that $\Bmu = (\mu', \mu'') \in \SP_{n,2}$.                                                   
\begin{enumerate}                                                                                   
\item                                                                                               
There exists an isomorphism of $S_n$-modules                                                        
                                                                                                    
\begin{equation*}                                                                                   
H^{2i}(\SB_{\Bmu})_{\vf_+} \simeq \begin{cases}                                                           
                            H^{i}(\SB^0_{\mu' + \mu''})                                             
                                &\quad\text{ if $i$ is even,}\\                                     
                            0   &\quad\text{ otherwise. }                                           
                         \end{cases}                                                                
\end{equation*}           
\item                                                                                               
$H^{2i}(\SB_{\Bmu})_{\vf_-} = 0$ unless $\Bmu = (-;\mu'')$.  Assume that                                  
$\Bmu = (-;\mu'')$.  There exists an isomorphism of $S_n$-modules                                   
\begin{equation*}                                                                                   
H^{2i}(\SB_{\Bmu})_{\vf_-} \simeq \begin{cases}                                                           
                            H^{i-n}(\SB^0_{\mu''})                                                  
                              &\quad\text{ if $i \equiv n \pmod 2$, } \\                            
                            0                                                                       
                              &\quad\text{ otherwise. }                                             
                            \end{cases}                                                             
\end{equation*}                                                                                     
\end{enumerate}                                                                                     
\end{cor}   

\begin{proof}
Assume that $\vf = \vf_+$.
In that case $(V_{\Bla})_{\vf} = 0$ unless                 
$\Bla = (\la';-)$, and in that case, $(V_{\Bla})_{\vf} \simeq V_{\la'}$ as $S_n$-modules.                 
Moreover, if $\Bla = (\la';-)$, we have                                                             
$\wt K_{\Bla, \Bmu}(t) = \wt K_{\la',\mu' + \mu''}(t^2)$ by Proposition 2.5 (ii).
On the other hand, assume that $\vf = \vf_-$.
Then we have
$(V_{\Bla})_{\vf} = 0$ unless $\Bla = (-;\la'')$, and in that case,                                      
$(V_{\Bla})_{\vf} \simeq V_{\la''}$ as $S_n$-modules.                                                     
Moreover, by Proposition 2.5 (i),                                                                   
if $\Bla = (-;\la'')$, $\wt K_{\Bla,\Bmu}(t) = 0$ unless                                            
$\Bmu = (-;\mu'')$, and in that case, $\wt K_{\Bla, \Bmu}(t) = t^n\wt K_{\la'',\mu''}(t)$.  
Then the corollary follows from  (3.19.1) by a similar discussion as in the proof of 
Proposition 3.19. 
\end{proof}

\para{3.22.}
Recall that the Hall-Littlewood function $P_{\Bla}(x;t)$ is defined by
two types of variables $x^{(1)}, x^{(2)}$. Here we consider a specialization 
of those variables. We denote by $P_{\Bla}(x;t)|_{x = (y,y)}$ the 
function in $\vL[t]$ obtained by substituting $x^{(1)} = x^{(2)} = y$. 
We further consider the specialization of this function by putting $t = 1$, i.e., 
$P_{\Bla}(x;1)|_{x = (y,y)}$.  The following result shows that the behaviour
of $P_{\Bla}(x;t)$ at $t = 1$ is quite different from that of ordinary 
Hall-Littlewood functions (cf. Remark 2.8).

\begin{prop}  
Under the notation as above, we have
\begin{equation*}
       P_{\Bmu}(x;1)|_{x = (y,y)} = 
            \begin{cases}
                m_{\mu''}(y)   &\quad\text{ if $\Bmu = (-;\mu'')$, }\\
                0              &\quad\text{ otherwise. }
            \end{cases}
\end{equation*}
\end{prop}

\begin{proof}
Assume that $\Bmu = (-;\mu'')$.  
Since $P_{\Bmu}(x;t) = P_{\mu''}(x^{(2)};t)$ for $\Bmu = (-;\mu'')$ by Corollary 1.12, 
we have 
\begin{equation*}
\tag{3.23.1}
P_{\Bmu}(x;1)|_{x = (y,y)} = m_{\mu''}(y),
\end{equation*}
which shows the first equality.
\par
By (1.2.1) and (1.2.2), for any $\la \in \SP_n$, we have
\begin{equation*}
s_{\la}(y) = \sum_{\mu \in \SP_n}K_{\la,\mu}(1)m_{\mu}(y).
\end{equation*}
Also by substituting $t = 1$ in the formula (3.3.3) and by using (1.2.1), we have, 
for any partitions $\mu, \nu$,  

\begin{equation*}
m_{\mu}(y)m_{\nu}(y) = \sum_{\la \in \SP_n}f^{\la}_{\mu\nu}(1)m_{\la}(y).
\end{equation*}

Thus for $\Bla = (\la';\la'') \in \SP_{n,2}$, we have

\begin{align*}
\tag{3.23.2}
s_{\Bla}(x)|_{x = (y,y)} &= s_{\la'}(y)s_{\la''}(y)  \\
                         &= \sum_{\nu'}\sum_{\nu''}K_{\la',\nu'}(1)K_{\la'', \nu''}(1)
                              m_{\nu'}(y)m_{\nu''}(y) \\
                         &= \sum_{\mu'' \in \SP_n}m_{\mu''}(y)
             \sum_{\nu', \nu''}f^{\mu''}_{\nu'\nu''}(1)K_{\la',\nu'}(1)K_{\la'',\nu''}(1) \\
                         &= \sum_{\Bmu = (-, \mu'')}K_{\Bla,\Bmu}(1)m_{\mu''}(y).
\end{align*}
The last equality follows from (3.4.1).
On the other hand, by 1.6, we have

\begin{equation*}
\tag{3.23.3}
s_{\Bla}(x)|_{x = (y,y)} = \sum_{\Bmu \in \SP_{n,2}}K_{\Bla,\Bmu}(1)P_{\Bmu}(x;1)|_{x = (y,y)}.
\end{equation*}
Put $\SP_{n,2}' = \{ \Bmu = (\mu', \mu'') \in \SP_{n,2} \mid |\mu'| \ne 0 \}$.
Then (3.23.2) and (3.23.3), together with (3.23.1)  imply that 

\begin{equation*}
\tag{3.23.4}
\sum_{\Bmu \in \SP'_{n,2}}K_{\Bla,\Bmu}(1)P_{\Bmu}(x; 1)|_{x = (y,y)} = 0
\end{equation*}  
for any $\Bla \in \SP_{n,2}$.  
By Proposition 1.7, $K_{\Bla,\Bmu}(t) = 0$ unless $\Bmu \le \Bla$, and 
$K_{\Bla,\Bla}(t) = 1$.  Now the proposition follows from (3.23.4) 
by induction on the partial order $\le $ on $\SP'_{n,2}$. The proposition is proved.  
\end{proof}
\par\bigskip

\section {Hall bimodule }

\para{4.1.}
Before going into detalis on the Hall bimodule, we show a preliminary result. 
In this section we fix a total order on $\SP_{n,2}$ which is compatible with the partial order 
$\le$ on $\SP_{n,2}$.  
For $\Bnu = (\nu', \nu'') \in \SP_{n,2}$, 
put $R_{\Bnu}(x;t) = P_{\nu'}(x^{(1)},t^2)P_{\nu''}(x^{(2)}, t^2)$.
Then $\{ R_{\Bnu} \mid \Bnu \in \SP_{n,2}\}$ gives a basis of $\Xi^n[t]$.
Hence there exist polynomials $h^{\Bmu}_{\Bnu}(t) \in \BZ[t]$ such that

\begin{equation*}
\tag{4.1.1}
R_{\Bnu}(x;t) = 
      \sum_{\Bmu \in \SP_{n,2}}h^{\Bmu}_{\Bnu}(t)P_{\Bmu}(x;t).
\end{equation*}
The transition matrix between the bases $\{ s_{\Bla}\}$ and $\{R_{\Bnu}\}$ is lower 
unitriangular, (with respect to the fixed total order) 
and a similar result holds also for the bases 
$\{s_{\Bla}\}$ and $\{P_{\Bmu}\}$. Hence the transition matrix 
$(h^{\Bmu}_{\Bnu}(t))_{\Bmu, \Bnu \in \SP_{n,2}}$ between 
$\{R_{\Bnu}\}$ and $\{P_{\Bmu}\}$ is also lower unitriangular 
(we regard that the $\Bnu\Bmu$-entry is $h^{\Bmu}_{\Bnu}(t)$).  
The following formula is an analogue of the formula (3.3.4) relating the polynomials  
$f_{\mu\nu}^{\la}(t)$ with the Hall polynomials $g^{\la}_{\mu\nu}(t)$. 

\begin{prop}  
Let $g^{\Bmu}_{\Bnu}(t)$ be the polynomials given in Propostion 3.2.  Then 
\begin{equation*}
\tag{4.2.1}
h^{\Bmu}_{\Bnu}(t) = t^{a(\Bmu) - a(\Bnu)}g^{\Bmu}_{\Bnu}(t^{-2}).
\end{equation*}
In particular, the matrix $(g^{\Bmu}_{\Bnu}(t))_{\Bmu,\Bnu}$ is lower unitriangular. 
\end{prop}

\begin{proof}
For any $\Bla = (\la',\la'') \in \SP_{n,2}$, we have

\begin{align*}
s_{\Bla}(x) &= s_{\la'}(x^{(1)})s_{\la''}(x^{(2)}) \\
            &= \sum_{\nu'}K_{\la',\nu'}(t^2)P_{\nu'}(x^{(1)};t^2)
               \sum_{\nu''}K_{\la'',\nu''}(t^2)P_{\nu''}(x^{(2)};t^2)  \\ 
            &= \sum_{\nu',\nu''}K_{\la',\nu'}(t^2)K_{\la'',\nu''}(t^2)
                      \sum_{\Bmu \in \SP_{n,2}}h^{\Bmu}_{\Bnu}(t)P_{\Bmu}(x; t)  \\
            &= \sum_{\Bmu \in \SP_{n,2}}\biggl(
               \sum_{\nu', \nu''}K_{\la',\nu'}(t^2)K_{\la'',\nu''}(t^2)
                    h^{\Bmu}_{\Bnu}(t)\biggr)P_{\Bmu}(x;t).
\end{align*}
Since

\begin{equation*}
s_{\Bla}(x) = \sum_{\Bmu \in \SP_{n,2}}K_{\Bla, \Bmu}(t)P_{\Bmu}(x;t),
\end{equation*}
by comparing the coefficients of $P_{\Bmu}(x;t)$, we have

\begin{equation*}
\tag{4.2.2}
K_{\Bla, \Bmu}(t) = \sum_{\nu', \nu''}h^{\Bmu}_{\Bnu}(t)K_{\la',\nu'}(t^2)
                         K_{\la'',\nu''}(t^2).
\end{equation*}
On the other hand, if we notice that 
$K_{\la'',\nu''}(t^2)\ne 0$ only when $|\la''| = |\nu''|$,  
the formual (3.3.1) can be rewritten as 

\begin{equation*}
\tag{4.2.3}
K_{\Bla, \Bmu}(t) = \sum_{\nu',\nu''}t^{a(\Bmu) - a(\Bnu)}g^{\Bmu}_{\Bnu}(t^{-2})
                          K_{\la',\nu'}(t^2)K_{\la'', \nu''}(t^2).
\end{equation*}
Since $(K_{\la',\nu'}(t^2)K_{\la'',\nu''}(t^2))_{\Bla, \Bnu \in \SP_{n,2}}$
is a unitriangular matrix with respect to the partial order on $\SP_{n,2}$, 
the proposition is obtained by comparing (4.2.2) and (4.2.3). 
\end{proof}

\para{4.3.}
We keep the assumption in 3.1, in particular, $k$ is an algebraic closure of 
$\BF_q$. Based on the idea of Finkelberg-Ginzburg-Travkin [FGT], 
we introduce the Hall bimodule. 
Let $\Bla, \Bmu \in \SP_{n,2}, \a \in \SP_n$, and take $(x,v) \in \SO_{\Bla}$.
We define varieties
\begin{align*}
\SG^{\Bla}_{\a,\Bmu} = \{W \subset V &\mid W \text{ : $x$-stable subspace, } \\
                                   &x|_W \text{ : type } \a, (x|_{V/W}, v \, \text{(mod } W ))
                                      \text{ : type } \Bmu\}, \\   
\SG^{\Bla}_{\Bmu,\a} = \{W \subset V 
            &\mid W\text{ : $x$-stable subspace, } v \in W,  \\
                         &(x|_W, v)\text{ : type $\Bmu$ }, x|_{V/W}\text{ : type $\a$} \}.   
\end{align*}
If $(x,v) \in \SO_{\Bla}(\BF_q)$, those vatieties are defined over $\BF_q$, and one can 
consider the subsets of $\BF_q$-fixed points. 
Assuming that $(x,v) \in \SO_{\Bla}(\BF_q)$, we define integers $G^{\Bla}_{\a,\Bmu}(q)$ and 
$G^{\Bla}_{\Bmu, \a}(q)$ by 
\begin{equation*}
\tag{4.3.1}
G^{\Bla}_{\a,\Bmu}(q) = |\SG^{\Bla}_{\a, \Bmu}(\BF_q)|,  \quad 
 G^{\Bla}_{\Bmu,\a}(q) = |\SG^{\Bla}_{\Bmu,\a}(\BF_q)|.
\end{equation*}
Note that $G^{\Bla}_{\a,\Bmu}(q), G^{\Bla}_{\Bmu,\a}(q)$ are independent of the choice of 
$(x,v) \in \SO_{\Bla}(\BF_q)$.
It is clear from the definition that $G^{\Bla}_{\a,\Bmu}(q) = G^{\Bla}_{\Bmu,\a}(q) = 0$ 
unless $|\Bla| = |\a| + |\Bmu|$. 
In the case where $\Bla = (-;\la''), \Bmu = (-;\mu'')$, 
$G^{\Bla}_{\a,\Bmu}(q) = G^{\Bla}_{\Bmu, \a}(q)$ coincides with 
$g^{\la''}_{\mu'',\a}(q) = g^{\la''}_{\mu'',\a}|_{t = q}$, where 
$g^{\la''}_{\mu'',\a}$ is the original Hall polynomial given in 3.3.
\par
Put $\SP = \coprod_{n \ge 0}\SP_n$ and $\SP^{(2)} = \coprod_{n \ge 0}\SP_{n,2}$.
Recall that the definition of the Hall algebea $\SH$; $\SH$  is the free $\BZ[t]$-module with basis
$\{ \Fu_{\a} \mid \a \in \SP \}$, and the multiplication is defined by 
\begin{equation*}
\Fu_{\b}\Fu_{\g} = \sum_{\a \in \SP_n}g^{\a}_{\b,\g}(t)\Fu_{\a},
\end{equation*}
where $n = |\b| + |\g|$. 
$\SH$ is a commutative, associative algebra over $\BZ[t]$. We define the 
$\BZ$-algebra $\SH_q$ by  
$\SH_q = \BZ \otimes_{\BZ[t]}\SH$, under the specialization 
$\BZ[t] \to \BZ$, $t \mapsto q$.  
\par
We define a Hall bimodule $\SM_q$ as follows; 
$\SM_q$ is a free $\BZ$-module with basis $\{ \Fu_{\Bla} \mid \Bla \in \SP^{(2)} \}$.
We define actions (the left action and the rignt acton) of $\SH_q$ on $\SM_q$ by 
\begin{align*}
\tag{4.3.2}
\Fu_{\a}\Fu_{\Bmu} &= \sum_{\Bla \in \SP_{n,2}}G^{\Bla}_{\a,\Bmu}(q)\Fu_{\Bla}, \\
\tag{4.3.3}
\Fu_{\Bmu}\Fu_{\a} &= \sum_{\Bla \in \SP_{n,2}}G^{\Bla}_{\Bmu,\a}(q)\Fu_{\Bla},
\end{align*} 
where $n = |\a| + |\Bmu|$. 
Then $\SM_q$ turns out to be a $\SH_q$-bimodule, which is verified as follows;
for partitions $\b, \g$, and double partitions $\Bla, \Bmu$, we define a variety 

\begin{align*}
\SG^{\Bla}_{\b,\g; \Bmu} = \{ (W_1 &\subset W_2) \mid W_1, W_2 
             \text{ : $x$-stable subspaces of $V$, } \\
                               &x_{W_1} \text{: type $\b$, } 
                                  x_{W_2/W_1} \text{ : type $\g$, } 
                                     (x_{V/W_2}, v\text{ (mod $W_2$)}) \text{ : type $\Bmu$ } \}. 
\end{align*}
We compute the number $|\SG^{\Bla}_{\b,\g;\Bmu}(\BF_q)|$ in two different ways. 
Put $n = |\b| + |\g|$. Assume that $x_{W_2}$ has type $\a$.  Then the cardinality of 
such $W_2$ is given by $G^{\Bla}_{\a,\Bmu}(q)$.  For each $W_2$, the cardinality of $W_1$ 
is given by $g^{\a}_{\b,\g}(q)$.  It follows that 
\begin {equation*}
\tag{4.3.4}
|\SG^{\Bla}_{\b,\g;\Bmu}(\BF_q)| = \sum_{\a \in \SP_n}g^{\a}_{\b,\g}(q)G^{\Bla}_{\a,\Bmu}(q).
\end{equation*}
On the other hand, the cardinality of $W_1$ satisfying the condition that 
$x|_{W_1}$ has type $\b$, $(x|_{V/W_1}, v\text{ (mod $W_1$)})$ has type $\Bnu$ is 
$G^{\Bla}_{\b,\Bnu}(q)$. For each $W_1$, the cardinality of $W_2$ such that 
$W_1 \subset W_2 \subset V$ and that $x|_{W_2/W_1}$ has type $\g$,  
$(x|_{V/W_2}, v\text{ (mod $W_2$)})$ has type $\Bmu$ is given by $G^{\Bnu}_{\g, \Bmu}(q)$.
It follows that 
\begin{equation*}
\tag{4.3.5}
|\SG^{\Bla}_{\b,\g;\Bmu}(\BF_q)| = \sum_{\Bnu \in \SP_{m,2}}
                     G^{\Bla}_{\b, \Bnu}(q)G^{\Bnu}_{\g,\Bmu}(q),
\end{equation*}
where $m = |\Bla| - |\b|$. 
Now the equality (4.3.4) $=$ (4.3.5) implies that 
$\Fu_{\b}(\Fu_{\g}\Fu_{\Bmu}) = (\Fu_{\b}\Fu_{\g})\Fu_{\Bmu}$.  
In a similar way, one can show that 
$(\Fu_{\Bmu}\Fu_{\g})\Fu_{\b} = \Fu_{\Bmu}(\Fu_{\g}\Fu_{\b})$. 
Next we consider a variety

\begin{equation*}
\begin{split}
\SG^{\Bla}_{\a;\Bmu;\b} = \{(&W_1 \subset W_2) \mid W_1, W_2 \text{ $x$-stable subspaces of $V$},
                              v \in W_2  \\
                     &x|_{W_1} \text{ : type $\a$, }, (x_{W_2/W_1}, v\text{ (mod $W_1$)})
                          \text{ : type $\Bmu$, } x_{V/W_2} \text{ : type $\b$} \}.
\end{split}
\end{equation*}
We compute the number $|\SG^{\Bla}_{\a;\Bmu;\b}(\BF_q)|$ in two different ways.
Take $W_2 \in \SG^{\Bla}_{\Bnu,\b}(\BF_q)$ for some $\Bnu \in \SP_{n,2}$ with 
$n = |\Bla| - |\b|$.  The cardinality of such $W_2$ is $G^{\Bla}_{\Bnu,\b}(q)$.
For each $W_2$, the cardinality of $W_1$ such that 
$(W_1 \subset W_2) \in \SG^{\Bla}_{\a;\Bmu;\b}(\BF_q)$ is given by $G^{\Bnu}_{\a,\Bmu}(q)$.
Thus
\begin{equation*}
|\SG^{\Bla}_{\a;\Bmu;\b}(\BF_q)| = \sum_{\Bnu \in \SP_{n,2}}
                           G^{\Bla}_{\Bnu,\b}(q)G^{\Bnu}_{\a,\Bmu}(q). 
\end{equation*}
On the other hand, first we take $W_1 \in \SG^{\Bla}_{\a,\Bnu}(\BF_q)$, and then take 
$W_2$ such that $(W_1 \subset W_2)$ is contained in $\FS^{\Bla}_{\a;\Bmu;\b}(\BF_q)$. 
This implies that 
\begin{equation*}
|\SG^{\Bla}_{\a;\Bmu;\b}(\BF_q)| = \sum_{\Bnu \in \SP_{n',2}}
                G^{\Bla}_{\a,\Bnu}(q)G^{\Bnu}_{\Bmu,\b}(q),
\end{equation*}
where $n' = |\Bla| - |\a|$. 
Comparing these two equalities, we have 
$\Fu_{\a}(\Fu_{\Bmu}\Fu_{\b}) = (\Fu_{\a}\Fu_{\Bmu})\Fu_{\b}$.
Thus $\SM_q$ has a structure of $\SH_q$-bimodule. 
\par
For an integer $n \ge 0$, let $\SM_q^n$ be the $\BZ$-submodule of $\SM_q$ 
spanned by $\Fu_{\Bla}$ with $\Bla \in \SP_{n,2}$.  Then we have 
$\SM_q = \bigoplus_{n \ge 0}\SM_q^n$. Similarly, we have a decompositon 
$\SH_q = \bigoplus_{n \ge 0}\SH_q^n$.
The above discussion shows that $\SM_q$ has a structure of graded $\SH_q$-bimodule,
i.e., $\SH_q^m\SM_q^n \subset \SM_q^{n+m}$, and 
$\SM_q^n\SH_q^m \subset \SM_q^{n+m}$.  

\para{4.4.}
For $\Bla = (-;-)$, put $\Fu_0 = \Fu_{\Bla}$.  It is easy to see that 
$\Fu_0\Fu_{\b} = \Fu_{(-;\b)}$ for $\b \in \SP$
(but $\Fu_{\b}\Fu_0 \ne \Fu_{(\b;-)}$). 
Take $\a, \b \in \SP$.  One can check that 
$G^{\Bla}_{\a; (-;\b)}(q) = g^{\Bla}_{(\a;\b)}(q)$ for $\Bla \in \SP^{(2)}$.
It follows, for $\a, \b \in \SP$,  that 

\begin{equation*}
\tag{4.4.1}
\Fu_{\a}\Fu_0\Fu_{\b} = \sum_{\Bla \in \SP_{n,2}}g^{\Bla}_{(\a;\b)}(q)\Fu_{\Bla},
\end{equation*}
where $n = |\a| + |\b|$. 
For each $\Bmu = (\mu';\mu'') \in \SP_{n,2}$, put $\Fv_{\Bmu} = \Fu_{\mu'}\Fu_0\Fu_{\mu''}$.
We have a lemma.

\begin{lem}  
$\{ \Fv_{\Bmu} \mid \Bmu \in \SP_{n,2} \}$ gives a basis of $\SM_q^n$.
Hence $\{ \Fv_{\Bmu} \mid \Bmu \in \SP^{(2)} \}$ gives a basis of $\SM_q$. 
For $\Bmu \in \SP_{n,2}$, we have 
\begin{equation*}
\tag{4.5.1}
\Fv_{\Bmu} = \sum_{\Bla \in \SP_{n,2}}g^{\Bla}_{\Bmu}(q)\Fu_{\Bla}.
\end{equation*}
In particular, $\SM_q$ is a free $\SH_q$-bimodule of rank 1 (with a basis 
$\Fv_{(-;-)} = \Fu_0$). 
\end{lem}

\begin{proof}
(4.5.1) follows from  (4.4.1).  $\SM_q^n$ is a free $\BZ$-module 
with rank $|\SP_{n,2}|$.  By Proposition 4.2, 
$(g^{\Bla}_{\Bmu}(q))_{\Bla, \Bmu \in \SP_{n,2}}$
is a unitriangular matrix with respect to a certain total order on $\SP_{n,2}$.   
Thus $\{ \Fv_{\Bmu} \mid \Bmu \in \SP_{n,2}\}$ gives rise to a basis of $\SM_q^n$.
\end{proof}

\para{4.6.}
Recall that $\Xi = \vL(x^{(1)}) \otimes \vL(x^{(2)})$, and 
$\Xi[t] = \vL(x^{(1)})[t]\otimes_{\BZ[t]}\vL(x^{(2)})[t]$.
Thus $\Xi[t]$ is regarded as a free $\vL[t]$-bimodule of rank 1 
($\vL = \vL(y)$ acts on $\vL(x^{(1)})$ by replacing $y$ by $x^{(1)}$, and 
so on for $\vL(x^{(2)})$). 
It is known by [M, III, (3.4)] that the map 
$\Fu_{\a} \mapsto t^{-n(\a)}P_{\a}(y;t\iv)$ gives an isomorphism 
of rings $\SH \otimes \BZ[t,t\iv] \isom  \vL\otimes \BZ[t,t\iv]$.  
This induces an isomorpism $\SH_q \otimes \BQ \isom \vL_{\BQ}$. 
We define a map $\Psi: \SM_{q^2} \otimes \BQ \to \Xi_{\BQ}$ by 
\begin{equation*}
\tag{4.6.1}
\Fv_{\Bmu} \mapsto  
   q^{-a(\Bmu)}P_{\mu'}(x^{(1)}, q^{-2})P_{\mu''}(x^{(2)}, q^{-2})
       = q^{-a(\Bmu)}R_{\Bmu}(x;q\iv)
\end{equation*}  
for $\Bmu = (\mu', \mu'') \in \SP^{(2)}$. 
Then it is clear that $\Psi$ gives an isomorphism 
$\SM_{q^2} \otimes \BQ \isom \Xi_{\BQ}$ of bimodules (under the isomorphism 
$\SH_{q^2} \otimes \BQ \isom \vL_{\BQ}$).
\par
By making use of (4.2.1), the formula (4.5.1) can be rewritten as 
\begin{equation*}
q^{a(\Bmu)}\Fv_{\Bmu} = \sum_{\Bla \in \SP_{n,2}}h^{\Bla}_{\Bmu}(q\iv)q^{a(\Bla)}\Fu_{\Bla},
\end{equation*}
where $\Fv_{\Bmu}, \Fu_{\Bla} \in \SM_{q^2}$. 
Since $(h^{\Bla}_{\Bmu}(q))_{\Bla, \Bmu \in \SP_{n,2}}$ is the transition matrix between 
the bases $\{R_{\Bmu}(x;q) \}$ and $\{ P_{\Bla}(x;q) \}$ of $\Xi^n_{\BQ}$, we see that 
\begin{equation*}
\tag{4.6.2}
\Psi(\Fu_{\Bla}) = q^{-a(\Bla)}P_{\Bla}(x;q\iv). 
\end{equation*}   
\par
For given $\Bla, \Bmu \in \SP^{(2)}$, $\a \in \SP$, we define polynomials 
$H^{\Bla}_{\a,\Bmu}(t), H^{\Bla}_{\Bmu,\a}(t) \in \BZ[t]$ by 

\begin{align*}
P_{\a}(x^{(1)};t^2)P_{\Bmu}(x;t) &= \sum_{\Bla \in \SP_{n,2}}H^{\Bla}_{\a,\Bmu}(t)P_{\Bla}(x;t), \\
P_{\Bmu}(x;t)P_{\a}(x^{(2)};t^2) &= \sum_{\Bla \in \SP_{n,2}}H^{\Bla}_{\Bmu,\a}(t)P_{\Bla}(x;t).
\end{align*}
where $n = |\a| + |\Bmu|$. 
Considering $\Psi\iv$, and by comparing (4.3.2) and (4.3.3), we have the following formula; 
for $\Bla, \Bmu \in \SP^{(2)}, \a \in \SP$, 

\begin{align*}
\tag{4.6.3}
G^{\Bla}_{\a,\Bmu}(q^2) &= q^{a(\Bla) - a(\Bmu) - 2n(\a)}H^{\Bla}_{\a,\Bmu}(q\iv), \\
\tag{4.6.4}
G^{\Bla}_{\Bmu,\a}(q^2) &= q^{a(\Bla) - a(\Bmu) - 2n(\a)}H^{\Bla}_{\Bmu,\a}(q\iv).
\end{align*}
The following result can be compared with that of the mirabolic Hall bimodule in [FGT, \S 4].

\begin{thm}  
Assume that $\Bla, \Bmu \in \SP^{(2)}, \a \in \SP$.  
\begin{enumerate}
\item
There exist polynomials 
$G^{\Bla}_{\a,\Bmu}, G^{\Bla}_{\Bmu,\a} \in \BZ[t]$ such that
$G^{\Bla}_{\a,\Bmu}(q) = G^{\Bla}_{\a,\Bmu}|_{t = q}$, 
$G^{\Bla}_{\Bmu,\a}(q) = G^{\Bla}_{\Bmu,\a}|_{t = q}$. 
Thus one can define a $\SH_t$-bimodule structure for the free $\BZ[t]$-module 
$\SM_t = \bigoplus_{\Bla \in \SP^{(2)}}\BZ[t]\Fu_{\Bla}$ by extending (4.3.2) and 
(4.3.3), 
where $\SH_t$ denotes the Hall algebra $\SH$ over $\BZ[t]$. 
\item
$\SM_t$ is a free $\SH_t$-bimodule of rank 1, with the basis $\Fu_0$.
More precisely, let $\{ \Fu_{\a} \mid \a \in \SP\}$ be the basis of $\SH_t$.
Then $\{ \Fu_{\mu'}\Fu_0\Fu_{\mu''} \mid (\mu'; \mu'') \in \SP^{(2)}\}$ gives 
a basis of $\SM_t$.  
For any $\Bmu = (\mu';\mu'') \in \SP_{n,2}$, we have
\begin{equation*}
\Fu_{\mu'}\Fu_0\Fu_{\mu''} = \sum_{\Bla \in \SP_{n,2}}g^{\Bla}_{\Bmu}(t)\Fu_{\Bla}.
\end{equation*}
\item
The map $\Psi: \Fu_{\Bla} \mapsto t^{-a(\Bla)}P_{\Bla}(x;t\iv)$ gives an isomorphism 
\begin{equation*}
\SM_{t^2}\otimes_{\BZ[t^2]}\BZ[t,t\iv] \isom \Xi\otimes\BZ[t, t\iv]
\end{equation*} 
as bimodules (under the isomorphism 
$\SH_{t^2}\otimes_{\BZ[t^2]}\BZ[t, t\iv] \simeq \vL\otimes \BZ[t, t\iv]$).
\end{enumerate}
\end{thm}

\begin{proof}
In view of (4.6.3) and (4.6.4), 
what we need to show is, for $\Bla, \Bmu \in \SP^{(2)}, \a \in \SP$,  
\begin{align*}
\tag{4.7.1}
&t^{a(\Bla) - a(\Bmu) - 2n(\a)}H^{\Bla}_{\a,\Bmu}(t\iv) \in \BZ[t^2], \\ 
\tag{4.7.2}
&t^{a(\Bla) - a(\Bmu) - 2n(\a)}H^{\Bla}_{\Bmu, \a}(t\iv) \in \BZ[t^2].
\end{align*}
All other assertions follow from the discussion in 4.6.
By (4.2.1), we see that $t^{a(\Bla) - a(\Bmu)}h^{\Bla}_{\Bmu}(t\iv) \in \BZ[t^2]$.
The matrix $H(t\iv) = (h^{\Bla}_{\Bmu}(t\iv))$ is unitriangular. Let $D(t)$ be the diagonal matrix
such that the $\Bla\Bla$-entry is $t^{a(\Bla)}$.  Then the matrix 
$(t^{a(\Bla) - a(\Bmu)}h^{\Bla}_{\Bmu}(t\iv))$ coincides with $D(t)\iv H(t\iv)D(t)$.
This matrix is also unitriangular. It follows that each entry of its inverse matrix is 
contained in $\BZ[t^2]$. Let $H(t\iv)\iv = (h'_{\Bmu,\Bnu}(t\iv))$ be the inverse matrix 
of $H(t\iv)$.  Then $t^{a(\Bnu)-a(\Bmu)}h'_{\Bmu,\Bnu}(t\iv) \in \BZ[t^2]$. 
Note that $H(t)$ is the transition matrix between the bases $\{ R_{\Bmu} \}$ and $\{ P_{\Bla} \}$. 
Hence $H(t)\iv$ is the transition matrix between the bases $\{ P_{\Bmu}\}$ and $\{ R_{\Bnu} \}$. 
One can write 
\begin{equation*}
P_{\Bmu}(x;t) = \sum_{\Bnu = (\nu',\nu'') \in \SP^{(2)}}h'_{\Bmu,\Bnu}(t)
        P_{\nu'}(x^{(1)};t^2)P_{\nu''}(x^{(2)};t^2).
\end{equation*}   
Since
\begin{equation*}
P_{\a}(x^{(1)};t^2)P_{\nu'}(x^{(1)};t^2) = \sum_{\xi \in \SP}
          f^{\xi}_{\a,\nu'}(t^2)P_{\xi}(x^{(1)};t^2),
\end{equation*}
we have 

\begin{align*}
P_{\a}(x^{(1)};t^2)P_{\Bmu}(x;t) &= \sum_{\Bnu \in \SP^{(2)}}h'_{\Bmu,\Bnu}(t)
          \sum_{\xi \in \SP}f^{\xi}_{\a,\nu'}(t^2)P_{\xi}(x^{(1)};t^2)P_{\nu''}(x^{(2)};t^2) \\
       &= \sum_{\Bnu, \xi}h'_{\Bmu,\Bnu}(t)f^{\xi}_{\a,\nu'}(t^2)
            \sum_{\Bla \in \SP^{(2)}}h^{\Bla}_{(\xi;\nu'')}(t)P_{\Bla}(x;t). 
\end{align*} 
It follows that 
\begin{equation*}
\tag{4.7.3}
H^{\Bla}_{\a,\Bmu}(t) = \sum_{\Bnu, \xi}h'_{\Bmu,\Bnu}(t)f^{\xi}_{\a,\nu'}(t^2)
           h^{\Bla}_{(\xi;\nu'')}(t).
\end{equation*}
Here $h'_{\Bmu,\Bnu}(t\iv) \in t^{a(\Bmu)-a(\Bnu)}\BZ[t^2]$ and 
$h^{\Bla}_{(\xi,\nu'')}(t\iv) \in t^{a((\xi;\nu'')) - a(\Bla)}\BZ[t^2]$.
Moreover, by (3.3.4), $f^{\xi}_{\a,\nu'}(t^{-2}) \in t^{2n(\a) + 2n(\nu')-2n(\xi)}\BZ[t^2]$. 
Since $a((\xi;\nu'')) = 2n(\xi) + 2n(\nu'') + |\nu''|$ and 
$a(\Bnu) = 2n(\nu') + 2n(\nu'') + |\nu''|$, we see that 
$H^{\Bla}_{\a,\Bmu}(t\iv) \in t^{a(\Bmu) + 2n(\a) - a(\Bla)}\BZ[t^2]$.   
This proves (4.7.1).  The proof of (4.7.2) is similar. 
\end{proof}
\par\bigskip
\begin{center}
{\sc Appendix  \ Tables of double Kostka polynomials }
\end{center}

\par\medskip
Let $K(t) = (K_{\Bla,\Bmu}(t))_{\Bla, \Bmu \in \SP_{n,2}}$ be the matrix of double Kostka 
polynomials. We give the table of matrices $K(t)$ for $2 \le n \le 5$.  
In the table below, we use the following notation; 
we denote the double partition $(\la;\mu)$ with 
$\la = (\la_1^{m_1}, \dots, \la_k^{m_k}), \mu = (\mu_1^{n_1}, \dots, \mu_{k'}^{n_{k'}})$ 
by 
$\la_1^{m_1}\cdots\la_k^{m_k}.\mu_1^{n_1}\cdots\mu_{k'}^{n_{k'}}$.  For example, 
\begin{equation*}
(21^2;3^2) \lra 21^2.3^2 \quad (32;-) \lra 32. \quad (-;21^2) \lra .21^2
\end{equation*}
and so on.  

\vspace*{1cm}
\begin{table}[h]
\caption{$K(t)$ for $n = 2$}
\label{B_2: B_2}
\begin{center}
\begin{tabular}{|c|ccccc|}
\hline
        &  $2.$  &  $1.1$  &  $.2$  &  $1^2.$  &  $.1^2$  \\
\hline 
 $2.$   &  1 &  $t$ &  $t^2$  & $t^2$ &  $t^4$ \\
 $1.1$  &    &  1   &  $t$    & $t$   &  $t^3 + t$ \\
 $.2  $ &    &      &  1      &       &  $t^2$   \\
 $1^2.$ &    &      &         &  1    &  $t^2$   \\
 $.1^2$ &    &      &         &       &   1  \\
\hline
\end{tabular}
\end{center}
\end{table}

\vspace{1cm}

\begin{table}[h]
\caption{$K(t)$ for $n = 3$}
\label{C_3: C_3}
\begin{center}
\begin{tabular}{|c|cccccccccc|}
\hline
      &    $3.$  &  $2.1$  &  $1.2$  &  $21.$  &  $1^2.1$  &  $.3$  &  $1.1^2$  &  $.21$  
                           &  $1^3.$  &  $.1^3$  \\ 
\hline 
 $3.$  &  1 &  $t$ &  $t^2$  & $t^2$ &  $t^3$  &   $t^3$  &  $t^4$    &   $t^5$  &  $t^6$  & $t^9$ \\
 $2.1$ &    &  1   &  $t$    & $t$   &  $t^2$  &   $t^2$  &  $t^3 + t$&   $t^4 + t^2$ & $t^5 + t^3$ 
                                                                      &  $t^8 + t^6 + t^4$   \\
 $1.2$ &    &      &  1      &       &  $t$    &   $t$    &  $t^2$     &  $t^3 + t$  &  $t^4$  
                                                                      &  $t^7 + t^5 + t^3$      \\
 $21.$ &    &      &         &  1    &  $t$    &          &  $t^2$     &  $t^3$  &  $t^4 + t^2$ 
                                                                      & $t^7 + t^5$    \\
 $1^2.1$ &    &      &         &       &   1     &          &  $t$  &  $t^2$  &  $t^3 + t$
                                                                           & $t^6 + t^4 + t^2$    \\
 $.3 $ &    &      &         &       &         &   1      &       & $t^2$  &     &   $t^6$  \\ 
 $1.1^2$ &    &      &         &       &         &          &   1   & $t$    &  $t^2$  
                                                                           & $t^5 + t^3 + t$ \\
 $.21$ &    &      &         &       &         &          &       &  1     &     &  $t^4 + t^2$ \\
 $1^3.$ &    &      &         &       &         &          &       &        &  1  &  $t^3$  \\
 $.1^3$ &    &      &         &       &         &          &       &        &     &   1   \\ 
\hline
\end{tabular}
\end{center}
\end{table}
\vspace{1cm}
\pagebreak

\footnotesize
\begin{table}[h]
\caption{$K(t)$ for $n = 4$}
\label{C_4: C_4}
\begin{center}
\hspace*{-0.1cm}
\begin{tabular}{|c|cccccccccccccc|}
\hline
      &  $4.$ &  $3.1$ &  $31.$ &  $2.2$ &  $21.1$ & $1.3$ &  $2.1^2$ &  
          $1^2.2$ &  $2^2.$ &  $1.21$ &  $.4$  &  $21^2.$ & $1^2.1^2$ & .31  \\
\hline 
$4.$ &  1  &  $t$ & $t^2$ & $t^2$ & $t^3$ & $t^3$ & $t^4$ & $t^4$ & $t^4$ & $t^5$ 
                                       & $t^4$  &  $t^6$  & $t^6$ &  $t^6$ \\ 
$3.1$ &     &   1  & $t$   & $t$   & $t^2$ & $t^2$ & $t^3 + t$ & $t^3$ & $t^3$ & $t^4 + t^2$ 
                                       & $t^3$  &  $t^5 + t^3$ & $t^5 + t^3$ & $t^5 + t^3$   \\ 
$31.$ &     &      &  1    &       & $t$   &       & $t^2$ & $t^2$  & $t^2$  &  $t^3$ 
                                       &       &   $t^4 + t^2$  & $t^4$ & $t^4$   \\ 
$2.2$ &     &      &       &   1   & $t$   & $t$   & $t^2$ & $t^2$  &  $t^2$ &  $t^3 + t$  
                                       &  $t^2$  &  $t^4$  & $t^4 + t^2$ &  $t^4 + t^2$       \\
$21.1$ &     &      &       &       &   1   &       & $t$   & $t$    &  $t$   & $t^2$  
                                       &       &  $t^3 + t$ & $t^3 + t$  &  $t^3$ \\  
$1.3$  &     &      &       &       &       &   1    &     &  $t$ &     &  $t^2$    
                                        &  $t$  &    &  $t^3$  &  $t^3 + t$ \\  
$2.1^2$ &    &      &       &       &       &        &   1 &      &     &  $t$  
                                        &       &   $t^2$   &  $t^2$  &  $t^2$ \\
$1^2.2$ &    &      &       &       &       &        &     &  1   &     &  $t$
                                        &       &           &  $t^2$  & $t^2$ \\
$2^2.$  &    &      &       &       &       &        &     &      &   1  &
                                        &       &    $t^2$   &  $t^2$  &   \\
$1.21$  &    &      &       &       &       &        &     &      &      &  1
                                        &       &            &   $t$  &  $t$  \\
$.4$    &    &      &       &       &       &        &     &      &      &
                                        &  1    &            &     & $t^2$    \\
$21^2.$  &   &      &       &       &       &        &     &      &      &
                                        &       &   1        &    &       \\  
$1^2.1^2$ &  &      &       &       &       &        &     &      &      &
      
                                        &       &            &   1   &     \\ 
$.31$    &   &      &       &       &       &        &     &       &      &
                                        &       &             &      &  1  \\
$1^3.1$  &&&&&&&&&&  &&&&  \\ 
$.2^2$   &&&&&&&&&&  &&&&   \\
$1.1^3$  &&&&&&&&&&  &&&&   \\
$.21^2$  &&&&&&&&&&  &&&&   \\
$1^4.$   &&&&&&&&&&  &&&&    \\
$.1^4$   &&&&&&&&&&  &&&&    \\
\hline 
\end{tabular}
\end{center}
\vspace{2cm}


\begin{center}
\hspace*{-0.7cm}
\begin{tabular}{|c|cccccc|}
\hline
         & $1^3.1$ & $.2^2$ & $1.1^3$ & $.21^2$
         &  $1^4.$ & $.1^4$  \\ 
\hline
$4.$  &  $t^7$ & $t^8$  &  $t^9$  &  $t^{10}$  & $t^{12}$  & $t^{16}$  \\
$3.1$ &  $t^6 + t^4$ &  $t^7 + t^5$
                           &  $t^8 + t^6 + t^4$  &  $t^9 + t^7 + t^5$  & $t^{11} + t^9 + t^7$  
                                  & $t^{15} + t^{13} + t^{11} + t^9$  \\
$31.$ &  $t^5 + t^3$  &  $t^6$  &  $t^7 + t^5$  &  $t^8 + t^6$ &  $t^{10} + t^8 + t^6$
                                  &  $t^{14} + t^{12} + t^{10}$  \\ 
$2.2$ &  $t^5 + t^3$ & $t^6 + t^4 + t^2$ & $ t^7 + t^5 + t^3$ & $t^8 + t^6 + 2t^4$
                           &$t^{10} + t^8 + t^6$ & $t^{14} + t^{12} + 2t^{10} + t^8 + t^6$ \\        
$21.1$ & $t^4 + 2t^2$ & $t^5 + t^3$ & $t^6 + 2t^4 + t^2$ & $t^7 + 2t^5 + t^3$ & 
                           $t^9 + 2t^7 + 2t^5 + t^3$ & $t^{13} + 2t^{11} + 2t^9 + 2t^7 + t^5$ \\   
$1.3$  & $t^4$  &  $t^5 + t^3$ &  $t^6$ &  $t^7 + t^5 + t^3$ &  $t^9$  & 
                            $t^{13} + t^{11} + t^9 + t^7$ \\  
$2.1^2$ &  $t^3$  &  $t^4$  &  $t^5 + t^3 + t$ &  $t^6 + t^4 + t^2$  &  $t^8 + t^6 + t^4$  & 
                            $t^{12} + t^{10} + 2t^8 + t^6 + t^4$ \\
$1^2.2$ & $t^3 + t$ & $t^4$  &  $t^5 + t^3$  &  $t^6 + t^4 + t^2$  &  $t^8 + t^6 + t^4$  &
                            $t^{12} + t^{10} + 2t^8 + t^6 + t^4$  \\
$2^2.$  & $t^3$  &  $t^4$  &  $t^5$   &  $t^6$   &  $t^8 + t^4$  &  $t^{12} + t^8$ \\ 
$1.21$  & $t^2$  &  $t^3 + t$  &  $t^4 + t^2$  &  $t^5 + 2t^3 + t$  &  $t^7 + t^5$   &
                             $t^{11} + 2t^9 + 2t^7 + 2t^5 + t^3$  \\
$.4$      &       &   $t^4$  &      &   $t^6$  &    &  $t^{12}$  \\
$21^2.$   &  $t$  &    &  $t^3$  &  $t^4$  &  $t^6 + t^4 + t^2$  &   $t^{10} + t^8 + t^6$ \\
$1^2.1^2$ &  $t$  &  $t^2$  &  $t^3 + t$ &  $t^4 + t^2$  &  $t^6 + t^4 + t^2$ & 
                              $t^{10} + t^8 + 2t^6 + t^4 + t^2$ \\
$.31$  &    &  $t^2$  &    &  $t^4 + t^2$  &     &  $t^{10} + t^8 + t^6$ \\ 
$1^3.1$ & 1  &    &  $t^2$  &  $t^3$  &  $t^5 + t^3 + t$ &  $t^9 + t^7 + t^5 + t^3$ \\
$.2^2$  &    &  1 &    &  $t^2$  &    &  $t^8 + t^4$ \\
$1.1^3$  &   &    &  1  &  $t$  &  $t^3$  &  $t^7 + t^5 + t^3 + t$ \\
$.21^2$  &   &    &     &  1    &      &   $t^6 + t^4 + t^2$  \\
$1^4.$   &   &    &     &       &   1  &   $t^4$  \\
$.1^4$   &   &    &     &       &      &    1    \\
\hline
\end{tabular}
\end{center}
\end{table}


\newpage
\normalsize
\begin{center}
{\sc Table 4.} $K(t)$ for $n = 5$ 
\end{center}
\footnotesize 
\par

\begin{sideways}
\begin{tabular}{|c|ccccccccccccccccccc|}
\hline
           &  $5.$     &  $4.1$  &  $3.2$   &  $41.$    &  $2.3$    
           &  $31.1$   &  $1.4$  &  $21.2$  &  $3.1^2$  &  $32.$  
           &  $1^2.3$  &  $2.21$ &  $2^2.1$ &  $1.31$   &  $21.1^2$  
           &  $31^2.$  &  $1^2.21$  &  $21^2.1$  &  $.5$   \\  
\hline 
    $5.$   &  1        &  $t$    &  $t^2$   &  $t^2$    &  $t^3$ 
           &  $t^3$    &  $t^4$  &  $t^4$   &  $t^4$    &  $t^4$ 
           &  $t^5$    &  $t^5$  &  $t^5$   &  $t^6$    &  $t^6$  
           &  $t^6$    &  $t^7$  &  $t^7$   &  $t^5$  \\  
    $4.1$  &           &  1           &  $t$     &  $t$         &  $t^2$ 
           &  $t^2$    &  $t^3$       &  $t^3$   &  $t^3 + t$   &  $t^3$
           &  $t^4$    &  $t^4 + t^2$ &  $t^4$   &  $t^5 + t^3$ &  $t^5 + t^3$  
           &  $t^5 + t^3$ &  $t^6 + t^4$  &  $t^6 + t^4$  &  $t^4$  \\ 
    $3.2$  &           &              &  1       &              &  $t$   
           &  $t$      &  $t^2$         &  $t^2$   &  $t^2$       &  $t^2$
           &  $t^3$    &  $t^3 + t$   &  $t^3$   &  $t^4 + t^2$ &  $t^4 + t^2$   
           &  $t^4$    &  $t^5 + t^3$ &  $t^5 + t^3$  &  $t^3$     \\
    $41.$  &           &              &          &  1           &  
           &  $t$      &              &  $t^2$   &  $t^2$       &  $t^2$
           &  $t^3$    &  $t^3$       &  $t^3$   &  $t^4$       &  $t^4$  
           &  $t^4 + t^2$ &  $t^5$    &  $t^5 + t^3$  &      \\
    $2.3$  &           &              &          &              &   1  
           &           &  $t$         &  $t$     &              &  
           &  $t^2$    &  $t^2$       &  $t^2$   &  $t^3 + t$   &  $t^3$   
           &           &  $t^4 + t^2$ &  $t^4$   &   $t^2$   \\
    $31.1$ &           &              &          &              &   
           &  1        &              &  $t$     &  $t$         &  $t$  
           &  $t^2$    &  $t^2$       &  $t^2$   &  $t^3$       &  $t^3 + t$  
           &  $t^3 + t$&  $t^4 + t^2$ &  $t^4 + 2t^2$   &     \\
    $1.4$  &           &              &          &              &
           &           &  1           &          &              &
           &  $t$      &              &          &  $t^2$       &   
           &           &  $t^3$       &          &  $t$   \\
    $21.2$ &           &              &          &              &
           &           &              &  1       &              &
           &  $t$      &  $t$         &  $t$     &  $t^2$       &    $t^2$  
           &           &  $t^3 + t$   &  $t^3 + t$   &      \\
   $3.1^2$ &           &              &          &              &
           &           &              &          &  1           &
           &           &  $t$         &          &  $t^2$       &  $t^2$ 
           &  $t^2$    &  $t^3$       &  $t^3$   &      \\
   $32.$   &           &              &          &              &
           &           &              &          &              &  1
           &           &              &   $t$    &              &   $t^2$  
           &  $t^2$    &  $t^3$       &   $t^3$  &        \\
   $1^2.3$ &           &              &          &              &
           &           &              &          &              &
           &  1        &              &          &   $t$        &  
           &           &  $t^2$       &          &    \\ 
   $2.21$  &           &              &          &              &
           &           &              &          &              &
           &           &  1           &          &   $t$        &   $t$ 
           &           &  $t^2$       &  $t^2$   &      \\ 
   $2^2.1$ &           &              &          &              &
           &           &              &          &              &
           &           &              &  1       &              &   $t$ 
           &           &  $t^2$       &  $t^2$   &      \\
   $1.31$  &           &              &          &              &
           &           &              &          &              &
           &           &              &          &  1           &     
           &           &   $t$        &          &     \\
 $21.1^2$  &           &              &          &              &
           &           &              &          &              &
           &           &              &          &              &   1  
           &           &   $t$        &   $t$    &      \\ 
 $31^2.$   &           &              &          &              &
           &           &              &          &              &
           &           &              &          &              &
           &   1       &              &  $t$     &      \\
 $1^2.21$  &           &              &          &              &
           &           &              &          &              &
           &           &              &          &              &
           &           &   1          &          &     \\
 $21^2.1$  &           &              &          &              &
           &           &              &          &              &
           &           &              &          &              &
           &           &              &   1      &     \\
 $.5$      &           &              &          &              &   
           &           &              &          &              &  
           &           &              &          &              &  
           &           &              &          &    1    \\ 
 $1.2^2$   &&&&&   &&&&&  &&&&&  &&&&  \\
 $2.1^3$   &&&&&   &&&&&  &&&&&  &&&&   \\
 $1^3.2$   &&&&&   &&&&&  &&&&&  &&&&   \\
 $2^21.$   &&&&&   &&&&&  &&&&&  &&&&   \\
 $.41$     &&&&&   &&&&&  &&&&&  &&&&   \\
 $1.21^2$  &&&&&   &&&&&  &&&&&  &&&&    \\
 $1^3.1^2$ &&&&&   &&&&&  &&&&&  &&&&    \\
 $.32$     &&&&&   &&&&&  &&&&&  &&&&    \\
 $21^3.$   &&&&&   &&&&&  &&&&&  &&&&    \\
 $1^2.1^3$ &&&&&   &&&&&  &&&&&  &&&&    \\
 $.31^2$   &&&&&   &&&&&  &&&&&  &&&&    \\
 $1^4.1$   &&&&&   &&&&&  &&&&&  &&&&   \\
 $.2^21$   &&&&&   &&&&&  &&&&&  &&&&  \\
 $1.1^4$   &&&&&   &&&&&  &&&&&  &&&&   \\
 $.21^3$   &&&&&   &&&&&  &&&&&  &&&&   \\
 $1^5.$    &&&&&   &&&&&  &&&&&  &&&&   \\
 $.1^5$    &&&&&   &&&&&  &&&&&  &&&&   \\
\hline
\end{tabular}
\end{sideways}

\begin{table}
\begin{sideways}
\begin{tabular}{|c|cccccccccc|}
\hline
           &  $1.2^2$   &  $2.1^3$    &  $1^3.2$    &  $2^21.$    
           &  $.41$      &  $1.21^2$  &  $1^3.1^2$  &  $.32$      &  $21^3.$  
           &  $1^2.1^3$  \\
\hline
 $5.$      &  $t^8$     &  $t^9$      &  $t^8$      &  $t^8$
           &  $t^7$      &  $t^{10}$  &  $t^{10}$   &  $t^9$      &  $t^{12}$
           &  $t^{11}$    \\  
 $4.1$     &  $t^7 + t^5$ &  $t^8 + t^6 + t^4$  &  $t^7 + t^5$  &  $t^7 + t^5$
           &  $t^6 + t^4$ &  $t^9 + t^7 + t^5$  &  $t^9 + t^7$  &  $t^8 + t^6$  &  $t^{11} + t^9 + t^7$
           &  $t^{10} + t^8 + t^6$   \\    
 $3.2$     &  $t^6 + t^4 + t^2$  &  $t^7 + t^5 + t^3$ & $t^6 + t^4$ & $t^6 + t^4$
           &  $t^5 + t^3$ &  $t^8 + t^6 + 2t^4$ &  $t^8 + t^6 + t^4$ & $t^7 + t^5 + t^3$ 
                                                                            & $t^{10} + t^8 + t^6$
           &  $t^9 + t^7 + 2t^5$     \\                
 $41.$     &  $t^6$              &  $t^7 + t^5$       &  $t^6 + t^4$  &  $t^6 + t^4$
           &  $t^5$       &  $t^8 + t^6$        &  $t^8 + t^6$       &  $t^7$   &  $ t^{10} + t^8 + t^6$
           &  $t^9 + t^7$    \\
 $2.3$     &  $t^5 + t^3$        &  $t^6$             &  $t^5 + t^3$  &   $t^5$
           &  $t^4 + t^2$ &  $t^7 + t^5 + t^3$  &  $t^7 + t^5$       &  $t^6 + t^4 + t^2$ &  $t^9$
           &  $t^8 + t^6 + t^4$   \\
 $31.1$    &  $t^5 + t^3$        &  $t^6 + 2t^4 + t^2$  &  $t^5 + 2t^3$  &  $t^5 + 2t^3$
           &  $t^4$       &  $t^7 + 2t^5 + t^3$ &  $t^7 + 2t^5 + t^3$  &  $t^6 + t^4$
                                                                          &  $t^9 + 2t^7 + 2t^5 + t^3$
           &  $t^8 + 2t^6 + 2t^4$  \\    
 $1.4$     &  $t^4$              &                      &  $t^4$         &
           &  $t^3 + t$          &  $t^6$               &  $t^6$         &   $t^5 + t^3$  
           &                     &  $t^7$   \\        
 $21.2$    &  $t^4 + t^2$        &  $t^5 + t^3$         &  $t^4 + 2t^2$  &   $t^4 + t^2$  
           &  $t^3$              &  $t^6 + 2t^4 + t^2$  &  $t^6 + 2t^4 + t^2  $ &  $t^5 + t^3$
           &  $t^8 + t^6 + t^4$  &  $t^7 + 2t^5 + 2t^3$  \\ 
 $3.1^2$   &  $t^4$              &  $t^5 + t^3 + t$     &  $t^4$         &  $t^4$ 
           &  $t^3$              &  $t^6 + t^4 + t^2$   &  $t^6$         & $t^5$  
           &  $t^8 + t^6 + t^4$  &  $t^7 + t^5 + t^3$     \\
 $32.$     &  $t^4$              &  $t^5$               &  $t^4$         &  $t^4 + t^2$
           &                     &  $t^6$               &  $t^6 + t^4$   &  $t^5$
           &  $t^8 + t^6 + t^4$  &  $t^7 + t^5$   \\
 $1^2.3$   &  $t^3$              &                      &  $t^3 + t$     &  
           &  $t^2$              &  $t^5 + t^3$         &  $t^5 + t^3$   &  $t^4$
           &                     &  $t^6 + t^4$     \\
 $2.21$    &  $t^3 + t$          &  $t^4 + t^2$         &  $t^3$         &  $t^3$  
           &  $t^2$              &  $t^5 + 2t^3 + t$    &  $t^5 + t^3$   &  $t^4 + t^2$ 
           &  $t^7 + t^5$        &  $t^6 + 2t^4 + t^2$   \\
 $2^2.1$   &  $t^3$              &  $t^4$               &  $t^3$         &  $t^3+ t$
           &                     &  $t^5$               &  $t^5 + t^3$   &  $t^4$
           &  $t^7 + t^5 + t^3$  &  $t^6 + t^4 + t^2$   \\
 $1.31$    &  $t^2$              &                      &  $t^2$         &
           &  $t$                &  $t^4 + t^2$         &  $t^4$         &  $t^3 + t$ 
           &                     &  $t^5 + t^3$    \\
 $21.1^2$  &  $t^2$              &  $t^3 + t$           &  $t^2$         &  $t^2$
           &                     &  $t^4 + t^2$         &  $t^4 + t^2$   &  $t^3$
           &  $t^6 + t^4 + t^2$  &  $t^5 + 2t^3 + t$   \\
 $31^2.$   &                     &  $t^3$               &  $t^2$         &  $t^2$
           &                     &  $t^4$               &  $t^4$         &  
           &  $t^6 + t^4 + t^2$  &  $t^5$    \\
 $1^2.21$  &  $t$                &                      &  $t$           &
           &                     &  $t^3 + t$           &  $t^3 + t$     &  $t^2$
           &                     &  $t^4 + 2t^2$    \\
 $21^2.1$  &                     &  $t^2$               &  $t$           &  $t$
           &                     &  $t^3$               &  $t^3 + t$     &  
           &  $t^5 + t^3 + t$    &  $t^4 + t^2$   \\
 $.5$      &                     &                      &                &
           &  $t^2$              &                      &                &   $t^4$
           &                     &   \\                     
 $1.2^2$   &  1                  &                      &                &
           &                     &   $t^2$              &  $t^2$         &   $t$
           &                     &   $t^3$   \\
 $2.1^3$   &                     &    1                 &                &
           &                     &    $t$               &                &
           &  $t^3$              &  $t^2$     \\
 $1^3.2$   &                     &                      &    1           &   
           &                     &  $t^2$               &   $t^2$        &
           &                     &  $t^3$   \\
 $2^21.$   &                     &                      &                &    1
           &                     &                      &   $t^2$        &
           &  $t^4 + t^2$        &  $t^3$   \\
 $.41$     &                     &                      &                &
           &   1                 &                      &                &   $t^2$
           &                     &     \\
 $1.21^2$  &                     &                      &                &
           &                     &      1               &                &
           &                     &  $t$  \\
 $1^3.1^2$ &                     &                      &                &
           &                     &                      &   1            &
           &                     &   $t$   \\
 $.32$     &                     &                      &                &
           &                     &                      &                &    1
           &                     &       \\
 $21^3.$   &                     &                      &                &
           &                     &                      &                &
           &  1                  &     \\
 $1^2.1^3$ &                     &                      &                &
           &                     &                      &                &
           &                     &    1    \\
 $.31^2$   &&&&  &&&&  &&   \\
 $1^4.1$   &&&&  &&&&  &&   \\
 $.2^21$   &&&&  &&&&  &&   \\
 $1.1^4$   &&&&  &&&&  &&    \\
 $.21^3$   &&&&  &&&&  &&    \\
 $1^5.$    &&&&  &&&&  &&     \\
 $.1^5$    &&&&  &&&&  &&   \\ 
\hline
\end{tabular}

\end{sideways}
\end{table}

\begin{table}
\vspace*{1cm}
\begin{sideways}
\begin{tabular}{|c|ccccc|}
\hline
           &  $.31^2$     &  $1^4.1$  &  $.2^21$   &  $1.1^4$    
           &  $.21^3$    \\
\hline
 $5.$      &  $t^{11}$     &  $t^{13}$    &  $t^{13}$    &  $t^{16}$
           &  $t^{17}$     \\       
 $4.1$     &  $t^{10} + t^8 + t^6$  &  $t^{12} + t^{10} + t^8$  &  $t^{12} + t^{10} + t^8$ 
                                                       &  $t^{15} + t^{13} + t^{11} + t^9$ 
           &  $t^{16} + t^{14} + t^{12} + t^{10}$   \\
 $3.2$     &  $t^9 + t^7 + 2t^5$  &  $t^{11} + t^9 + 2t^7$  & $t^{11} + t^9 + 2t^7 + t^5$
           &  $t^{14} + t^{12} + 2t^{10} + t^8 + t^6$  &  $t^{15} + t^{13} + 2t^{11} + 2t^9 + t^7$  \\
 $4.1$     &  $t^9 + t^7$                          &  $t^{11} + t^9 + t^7$  
           &  $t^{11} + t^9$                       &  $t^{14} + t^{12} + t^{10}$
           &  $t^{15} + t^{13} + t^{11}$           \\
 $41.$     &  $t^9 + t^7$                          &  $t^{11} + t^9 + t^7$  
           &  $t^{11} + t^9$                       &  $t^{14} + t^{12} + t^{10}$ 
           &  $t^{15} + t^{13} + t^{11}$           \\
 $2.3$     &  $t^8 + t^6 + 2t^4$                   &  $t^{10} + t^8 + t^6$  
           &  $t^{10} + t^8 + 2t^6 + t^4$          &  $t^{13} + t^{11} + t^9 + t^7$
           &  $t^{14} + t^{12} + 2t^{10} + 2t^8 + t^6$  \\    
 $31.1$    &  $t^8 + 2t^6 + t^4$                   &  $t^{10} + 2t^8 + 3t^6 + t^4$
           &  $t^{10} + 2t^8 + 2t^6$               &  $t^{13} + 2t^{11} + 3t^9 + 2t^7 + t^5$
           &  $t^{14} + 2t^{12} + 3t^{10} + 2t^8  + t^6$   \\
 $1.4$     &  $t^7 + t^5 + t^3$                    &  $t^9$
           &  $t^9 + t^7 + t^5$                    &  $t^{12}$  
           &  $t^{13} + t^{11} + t^9 + t^7$    \\
 $21.2$    &  $t^7 + 2t^5 + t^3$                   &  $t^9 + 2t^7 + 3t^5 + t^3$ 
           &  $t^9 + 2t^7 + 2t^5 + t^3$            &  $t^{12} + 2t^{10} + 3t^8 + 2t^6 + t^4$
           &  $t^{13} + 2t^{11} + 3t^9 + 3t^7 + 2t^5$   \\ 
 $3.1^2$   &  $t^7 + t^5 + t^3$                    &  $t^9 + t^7 + t^5$  
           &  $t^9 + t^7 + t^5$                    &  $t^{12} + t^{10} + 2t^8 + t^6 + t^4$
           &  $t^{13} + t^{11} + 2t^9 + t^7 + t^5$  \\
 $32.$     &  $t^7$                                &  $t^9 + t^7 + t^5$
           &  $t^9 + t^7$                          &  $t^{12} + t^{10} + t^8$
           &  $t^{13} + t^{11} + t^9$   \\
 $1^2.3$   &  $t^6 + t^4 + t^2$                    &  $t^8 + t^6 + t^4$
           &  $t^8 + t^6 + t^4$                    &  $t^{11} + t^9 + t^7$
           &  $t^{12} + t^{10} + 2t^8 + t^6 + t^4$  \\
 $2.21$    &  $t^6 + 2t^4 + t^2$                   &  $t^8 + 2t^6 + t^4$  
           &  $t^8 + 2t^6 + 2t^4 + t^2$            &  $t^{11} + 2t^9 + 2t^7 + 2t^5 + t^3$
           &  $t^{12} + 2t^{10} + 3t^8 + 3t^6 + 2t^4$  \\  
 $2^2.1$   &  $t^6$                                &  $t^8 + t^6 + 2t^4$
           &  $t^8 + t^6 + t^4$                    &  $t^{11} + t^9 + 2t^7 + t^5$
           &  $t^{12} + t^{10} + 2t^8 + t^6$    \\
 $1.31$    &  $t^5 + 2t^3 + t$                     &  $t^7 + t^5$
           &  $t^7 + 2t^5 + 2t^3$                  &  $t^{10} + t^8 + t^6$
           &  $t^{11} + 2t^9 + 3t^7 + 2t^5 + t^3$    \\
 $21.1^2$  &  $t^5 + t^3$                          &  $t^7 + 2t^5 + 2t^3$
           &  $t^7 + 2t^5 + t^3$                   &  $t^{10} + 2t^8 + 3t^6 + 2t^4 + t^2$
           &  $t^{11} + 2t^9 + 3t^7 + 2t^5 + t^3$   \\
 $31^2.$   &  $t^5$                                &  $t^7 + t^5 + t^3$
           &  $t^7$                                &  $t^{10} + t^8 + t^6$
           &  $t^{11} + t^9 + t^7$    \\  
 $1^2.21$  &  $t^4 + t^2$                          &  $t^6 + 2t^4 + t^2$
           &  $t^6 + 2t^4 + t^2$                   &  $t^9 + 2t^7 + 2t^5 + t^3$ 
           &  $t^{10} + 2t^8 + 3t^6 + 2t^4 + t^2$  \\
 $21^2.1$  &  $t^4$                                &  $t^6 + 2t^4 + 2t^2$
           &  $t^6 + t^4$                          &  $t^9 + 2t^7 + 2t^5 + t^3$
           &  $t^{10} + 2t^8 + 2t^6 + t^4$  \\
 $.5$      &  $t^6$                                &
           &  $t^8$                                &
           &  $t^{12}$      \\
 $1.2^2$   &  $t^3$                                &  $t^5$
           &  $t^5 + t^3 + t$                      &  $t^8 + t^4$
           &  $t^9 + t^7 + 2t^5 + t^3$  \\
 $2.1^3$   &  $t^2$                                &  $t^4$
           &  $t^4$                                &  $t^7 + t^5 + t^3 + t$
           &  $t^8 + t^6 + t^4 + t^2$     \\
 $1^3.2$   &  $t^3$                                &  $t^5 + t^3 + t$
           &  $t^5$                                &  $t^8 + t^6 + t^4$  
           &  $t^9 + t^7 + t^5 + t^3$     \\
 $2^21.$   &                                       &  $t^5 + t^3$ 
           &  $t^5$                                &  $t^8 + t^6$  
           &  $t^9 + t^7$    \\
 $.41$     &  $t^4 + t^2$                          &  
           &  $t^6 + t^4$                          &  
           &  $t^{10} + t^8 + t^6$   \\
 $1.21^2$  &  $t$                                  &  $t^3$
           &  $t^3 + t$                            &  $t^6 + t^4 + t^2$
           &  $t^7 + 2t^5 + 2t^3 + t$  \\
 $1^3.1^2$ &                                       &  $t^3 + t$ 
           &  $t^3$                                &  $t^6 + t^4 + t^2$
           &  $t^7 + t^5 + t^3$    \\
 $.32$     &  $t^2$                                &  
           &  $t^4 + t^2$                          &
           &  $t^8 + t^6 + t^4$     \\
 $21^3.$   &                                       &  $t$
           &                                       &  $t^4$
           &  $t^5$   \\
 $1^2.1^3$ &                                       &  $t^2$  
           &  $t^2$                                &  $t^5 + t^3 + t$  
           &  $t^6 + t^4 + t^2$   \\
 $.31^2$   &  1                                    &  
           &  $t^2$                                &
           &  $t^6 + t^4 + t^2$    \\
 $1^4.1$   &                                       &  1
           &                                       &  $t^3$
           &  $t^4$   \\
 $.2^21$   &                                       &
           &   1                                   &                   
           &  $t^4 + t^2$   \\
 $1.1^4$   &                                       &
           &                                       &   1
           &  $t$    \\
 $.21^3$   &                                       &
           &                                       &
           &  1  \\
 $1^5.$    &&&&&  \\
 $.1^5$    &&&&&   \\
\hline 
\end{tabular}
\end{sideways}
\vspace{1cm}
\end{table}

\begin{table}
\vspace*{1cm}
\begin{sideways}
\begin{tabular}{|c|cc|}
\hline
           &  $1^5.$     &  $.1^5$ \\
\hline
 $5.$      &  $t^{20}$     
           &  $t^{25}$   \\       
 $4.1$     &  $t^{19} + t^{17} + t^{15} + t^{13}$
           &  $t^{24} + t^{22} + t^{20} + t^{18} + t^{16}$  \\
 $3.2$     &  $t^{18} + t^{16} + 2t^{14} + t^{12} + t^{10}$
           &  $t^{23} + t^{21} + 2t^{19} + 2t^{17} + 2t^{15} + t^{13} + t^{11}$ \\ 
 $41.$     &  $t^{18} + t^{16} + t^{14} + t^{12}$ 
           &  $t^{23} + t^{21} + t^{19} + t^{17}$  \\
 $2.3$     &  $t^{17} + t^{15} + t^{13} + t^{11}$
           &  $t^{22} + t^{20} + 2t^{18} + 2t^{16} + 2t^{14} + t^{12} + t^{10}$  \\  
 $31.1$    &  $t^{17} + 2t^{15} + 3t^{13} + 3t^{11} + 2t^9 + t^7$
           &  $t^{22} + 2t^{20} + 3t^{18} + 3t^{16} + 3t^{14} + 2t^{12} + t^{10}$  \\
 $1.4$     &  $t^{16}$  
           &  $t^{21} + t^{19} + t^{17} + t^{15} + t^{13}$  \\
 $21.2$    &  $t^{16} + 2t^{14} + 3t^{12} + 3t^{10} + 2t^8 + t^6$  
           &  $t^{21} + 2t^{19} + 3t^{17} + 4t^{15} + 4t^{13} + 3t^{11} + 2t^9 + t^7$  \\
 $3.1^2$   &  $t^{16} + t^{14} + 2t^{12} + t^{10} + t^8$
           &  $t^{21} + t^{19} + 2t^{17} + 2t^{15} + 2t^{13} + t^{11} + t^9$   \\
 $32.$     &  $t^{16} + t^{14} + t^{12} + t^{10} + t^8$  
           &  $t^{21} + t^{19} + t^{17} + t^{15} + t^{13}$ \\
 $1^2.3$   &  $t^{15} + t^{13} + t^{11} + t^9$
           &  $t^{20} + t^{18} + 2t^{16} + 2t^{14} + 2t^{12} + t^{10} + t^8$  \\    
 $2.21$    &  $t^{15} + 2t^{13} + 2t^{11} + 2t^9 + t^7$
           &  $t^{20} + 2t^{18} + 3t^{16} + 4t^{14} + 4t^{12} + 3t^{10} + 2t^8 + t^6$  \\
 $2^2.1$   &  $t^{15} + t^{13} + 2t^{11} + 2t^9 + t^7 + t^5$
           &  $t^{20} + t^{18} + 2t^{16} + 2t^{14} + 2t^{12} + t^{10} + t^8$  \\
 $1.31$    &  $t^{14} + t^{12} + t^{10}$
           &  $t^{19} + 2t^{17} + 3t^{15} + 3t^{13} + 3t^{11} + 2t^9 + t^7$  \\
 $21.1^2$  &  $t^{14} + 2t^{12} + 3t^{10} + 3t^8 + 2t^6 + t^4$
           &  $t^{19} + 2t^{17} + 3t^{15} + 4t^{13} + 4t^{11} + 3t^9 + 2t^7 + t^5$ \\
 $31^2.$   &  $t^{14} + t^{12} + 2t^{10} + t^8 + t^6$ 
           &  $t^{19} + t^{17} + 2t^{15} + t^{13} + t^{11}$  \\
 $1^2.21$  &  $t^{13} + 2t^{11} + 2t^9 + 2t^7 + t^5$
           &  $t^{18} + 2t^{16} + 3t^{14} + 4t^{12} + 4t^{10} + 3t^8 + 2t^6 + t^4$ \\
 $21^2.1$  &  $t^{13} + 2t^{11} + 3t^9 + 3t^7 + 2t^5 + t^3$
           &  $t^{18} + 2t^{16} + 3t^{14} + 3t^{12} + 3t^{10} + 2t^8 + t^6$  \\ 
 $.5$      &      
           &  $t^{20}$    \\
 $1.2^2$   &  $t^{12} + t^8$
           &  $t^{17} + t^{15} + 2t^{13} + 2t^{11} + 2t^9 + t^7 + t^5$  \\
 $2.1^3$   &  $t^{11} + t^9 + t^7 + t^5$
           &  $t^{16} + t^{14} + 2t^{12} + 2t^{10} + 2t^8 + t^6 + t^4$  \\
 $1^3.2$   &  $t^{12} + t^{10} + 2t^8 + t^6 + t^4$ 
           &  $t^{17} + t^{15} + 2t^{13} + 2t^{11} + 2t^9 + t^7 + t^5$ \\
 $2^21.$   &  $t^{12} + t^{10} + t^8 + t^6 + t^4$
           &  $t^{17} + t^{15} + t^{13} + t^{11} + t^9$  \\
 $.41$     &  
           &  $t^{18} + t^{16} + t^{14} + t^{12}$  \\
 $1.21^2$  &  $t^{10} + t^8 + t^6$  
           &  $t^{15} + 2t^{13} + 3t^{11} + 3t^9 + 3t^7 + 2t^5 + t^3$  \\
 $1^3.1^2$ &  $t^{10} + t^8 + 2t^6 + t^4 + t^2$
           &  $t^{15} + t^{13} + 2t^{11} + 2t^9 + 2t^7 + t^5 + t^3$  \\    
 $.32$     &  
           &  $t^{16} + t^{14} + t^{12} + t^{10} + t^8$  \\
 $21^3.$   &  $t^8 + t^6 + t^4 + t^2$
           &  $t^{13} + t^{11} + t^9 + t^7$  \\
 $1^2.1^3$ &  $t^9 + t^7 + t^5 + t^3$
           &  $t^{14} + t^{12} + 2t^{10} + 2t^8 + 2t^6 + t^4 + t^2$  \\ 
 $.31^2$   &  
           &  $t^{14} + t^{12} + 2t^{10} + t^8 + t^6$  \\
 $1^4.1$   &  $t^7 + t^5 + t^3 + t$
           &  $t^{12} + t^{10} + t^8 + t^6 + t^4$  \\
 $.2^21$   &  
           &  $t^{12} + t^{10} + t^8 + t^6 + t^4$  \\
 $1.1^4$   &  $t^4$
           &  $t^9 + t^7 + t^5 + t^3 + t$  \\
 $.21^3$   &  
           &  $t^8 + t^6 + t^4 + t^2$   \\
 $1^5.$    &  1
           &  $t^5$   \\
 $.1^5$    &  
           &  1   \\  
\hline
\end{tabular}
\end{sideways}
\end{table}

\par\vspace{1cm}
\noindent  
L. Shiyuan \\  
Department of Mathematics, Tongji University \\ 
1239 Siping Road, Shanghai 200092 
E-mail: \verb|liushiyuantj@sina.com|                                                            

\par\medskip
\noindent
T. Shoji \\  
Department of Mathematics, Tongji University \\ 
1239 Siping Road, Shanghai 200092 
E-mail: \verb|shoji@tongji.edu.cn|                                                            
 \end{document}